\pdfoutput=1
\RequirePackage[figuresleft]{rotating}

\documentclass{amsart}
\usepackage{amsmath,amssymb,thmtools,amscd}
\usepackage{mathtools}
\usepackage{subcaption}

\usepackage[utf8]{inputenc}

\def \A {\mathcal{A}}
\def \B {\mathcal{B}}
\def \C {\mathcal{C}}

\newcommand{\ol}[1]{\smash{\overline{#1}}\vphantom{#1}}

\def \ra {\rightarrow}

\DeclarePairedDelimiterX{\gset}[2]{\{}{\}}{\,#1:#2\,}

\usepackage{tikz}
\usetikzlibrary{calc}
\usetikzlibrary{decorations.pathreplacing, positioning}

\usetikzlibrary{automata, shapes.geometric, shapes.misc, shapes.symbols, arrows.meta}

\usetikzlibrary{backgrounds, fit, matrix, math} 
\tikzstyle{background}=[rectangle, fill=gray!10, inner sep=-0.02cm, rounded corners=0.7mm] 
\usepackage{adjustbox} 
\usepackage{tikz-cd}
\tikzcdset{arrow style=tikz, diagrams={>=stealth}}
\usepackage[nomessages]{fp}

\newcommand \drawArrows[3][m]{
    \foreach \i in {#2} {
      \foreach \j in {#3} {
        \draw[->] let
          \n1 = {int(\j-1)}, 
          \n2 = {int(\j+1)}  
        in
          (#1-\i-\n1.east) -> (#1-\i-\n2.west);
        \draw[->] let
          \n1 = {int(\i-1)}, 
          \n2 = {int(\i+1)}  
        in
          (#1-\n1-\j.south) -> (#1-\n2-\j.north);
      };
    };
}

\tikzset{
  dominopicture/.style={
    baseline=0mm,
    dominoinode/.style={
      font=\small,
      inner sep=0mm,
      outer sep=0mm,
      minimum height=2.5mm,
      minimum width=2.5mm,
      execute at begin node=\strut
    },
  },
  smalldominopicture/.style={
    baseline=-.75mm,
    dominoinode/.style={
      font=\small,
      inner sep=0mm,
      outer sep=0mm,
      minimum height=2mm,
      minimum width=1mm,
    },
  },
}

\newcommand\domino[3][dominopicture]{
  \begin{tikzpicture}[#1]
    \node[anchor=base,dominoinode] (lnode) at (0,0) {$#2$};
    \node[anchor=base west,dominoinode] (rnode) at ($ (lnode.base
east) + (1mm,0) $) {$#3$};
    \coordinate (nw) at ($ (current bounding box.north west) + (-1mm,.15mm) $);
    \coordinate (se) at ($ (current bounding box.south east) + (1mm,-.15mm) $);
    \coordinate (mid) at ($ (lnode.east)!.5!(rnode.west) $);
    \draw[-] ($ (mid |- nw) $)--($ (mid |- se) $);
    \draw[-] (nw) rectangle (se);
  \end{tikzpicture}
}

\newtheorem{theorem}{Theorem}[section]
\newtheorem{proposition}[theorem]{Proposition}

\newtheorem{fact}[theorem]{Fact}
\newtheorem{corollary}[theorem]{Corollary}

\newtheorem{example}[theorem]{Example}
\newtheorem{question}[theorem]{Question}
\newtheorem{conjecture}[theorem]{Conjecture}

\theoremstyle{definition}

\newtheorem{remark}[theorem]{Remark}

\usepackage{hyperref}

\usepackage{cases}

\usepackage{amsaddr}

\begin{document}

\title{Preserving self-similarity in free products of semigroups}

\author{Tara Macalister Brough}
\address{Centro de Matem\'{a}tica e Aplica\c{c}\~{o}es, Faculdade de Ci\^{e}ncias e Tecnologia \\
Universidade Nova de Lisboa, 2829--516 Caparica, Portugal}

\author{Jan Philipp Wächter}
\address{Department of Mathematics,
  University of Manchester\\
  Oxford Road,
  Manchester M13 9PL, UK}

\author{Janette Welker}
\address{Theoretical Computer Science Group,
  Goethe University Frankfurt\\
  Robert-Mayer-Str.\ 11-15,
  60325 Frankfurt am Main, Germany}

\email{tarabrough@gmail.com}
\email{j.ph.waechter@gmail.com} 
\email{welker@em.uni-frankfurt.de} 

\maketitle

\begin{abstract}
  We improve on earlier results on the closure under free products of the class of automaton semigroups. We consider partial automata and show that the free product of two self-similar semigroups (or automaton semigroups) is self-similar (an automaton semigroup) if there is a homomorphism from one of the base semigroups to the other. The construction used is computable and yields further consequences. One of them is that we can adjoin a free generator to any self-similar semigroup (or automaton semigroup) and preserve the property of self-similarity (or being an automaton semigroup).

  The existence of a homomorphism between two semigroups is a very lax requirement; in particular, it is satisfied if one of the semigroups contains an idempotent. To explore the limits of this requirement, we show that no simple or $0$-simple idempotent-free semigroup is a finitely generated self-similar semigroup (or an automaton semigroup). Furthermore, we give an example of a pair of residually finite semigroups without a homomorphism from one to the other.
\end{abstract}


\section{Introduction}

The problem of presenting (finitely generated) free groups and semigroups in a self-similar way has a long history \cite{rodaro2022selfSimilarity}. A self-similar presentation in this context is typically a faithful action on an infinite regular tree (with finite degree) such that, for any element and any node in the tree, the action of the element on the (full) subtree rooted at the node is the same as that of a (possibly different) element on the entire tree (i.\,e.\ at the root). The idea for the name here is that the action on a full subtree is similar to the action of the group or semigroup on the entire tree. An important special case of such a self-similar presentation occurs when there is a finite set of generators such that the action of any generator on the subtree below any node is the same as the action of some (potentially different) generator at the root. By identifying the nodes of the infinite regular tree with the strings over an appropriate finite alphabet, we can describe such an action using a finite automaton (more precisely, a finite-state letter-to-letter -- or synchronous -- transducer), which leads to the class of \emph{automaton semigroups} and \emph{automaton groups} (also often called `automata groups'). If we relax the finite-state requirement and also consider infinite automata, we can even describe any self-similar action in this way. This is the approach we will take in this paper.

There is a quite interesting evolution of constructions to present free groups in a self-similar way or even as automaton groups (see \cite{rodaro2022selfSimilarity} for an overview). This culminated in constructions to present free groups of arbitrary rank as automaton groups where the number of states coincides with the rank \cite{vorobets2010series, steinberg2011automata}. While these constructions and the involved proofs are generally deemed quite complicated, the situation for semigroups turns out to be much simpler. While it is known that the free semigroup of rank one is not an automaton semigroup \cite[Proposition~4.3]{cain_1auto}, the free semigroups of higher rank can be generated by an automaton \cite{grigorchuk2001lamplighter, silva2005class} (see also \cite[Proposition~4.1]{cain_1auto}). In fact, the construction to generate these semigroups is quite simple \cite[Proposition~4.1]{cain_1auto} (compare also to \autoref{ex:freeSemigroup}). The same construction can also be used to generate free monoids as automaton semigroups or monoids. Here, the main difference is that the free monoid in one generator can indeed be generated by an automaton: it is generated by the adding machine (see \autoref{ex:addingMachine}), which also generates the free group of rank one if inverses are added. On a side note, it is also worthwhile to point out that -- although there does not seem to be much research on the topic -- there are examples to generate the free inverse semigroup of rank one as a subsemigroup of an automaton semigroup \cite[Theorem~25]{Oliynyk2010inverse} and an adaption to present the free inverse monoid of rank one as an automaton semigroup \cite[Example~2]{decidabilityPart} (see also \cite[Example~23]{structurePart} and \cite{kochubinska2024monogenic}).

While the question which free groups and semigroups can be generated using automata is settled, there is a related natural question, which is still open: is the free pro\-duct of two automaton/self-similar (semi)groups again an automaton/self-similar (semi)group? The free product of two groups or semigroups $X = \langle P \mid \mathcal{R} \rangle$ and $Y = \langle Q \mid \mathcal{S} \rangle$ is the group or semigroup $X \star Y = \langle P \cup Q \mid \mathcal{R}\cup \mathcal{S} \rangle$. Here it is very important to make the distinction whether we consider these presentations to be semigroup presentations (i.\,e.\ we work in the category of semigroups) or group presentations (i.\,e.\ we work in the category of groups). In particular, the free product in the sense of semigroups of two groups is not a group (as it is not even a monoid).

There are quite a few results on free (and related) products of self-similar or automaton groups (again see \cite{rodaro2022selfSimilarity} for an overview) 
but many of them present the product as a subgroup of an automaton/self-similar group and, thus, lose the self-similarity property. An exception here is a line of 
research based on the Bellaterra automaton which resulted in a construction to generate the free product of an arbitrary number of copies of the group of order two as an 
automaton group \cite{savchuk2011automata} (see also \cite{steinberg2011automata}).
However, there do not seem to be constructions for presenting arbitrary free products of self-similar groups in a self-similar way.

For semigroups, the situation is much better understood.  In fact, the free product of two automaton semigroups $S$ and $T$ is always at least very close to being an 
automaton semigroup: adjoining an identity to $S\star T$ results in an automaton semigroup \cite[Theorem~3]{bc_automaton1}. 
Thus, if we drop the self-similarity requirement and consider subsemigroups of automaton semigroups (sometimes called \emph{semigroups generated by automata}), the closure of this class under free products is assured.
Recently, it was even shown that the class of semigroups generated by automata over a fixed alphabet is closed under free product \cite{kochubinska2025free}.

Unlike for groups, there are also very general results on free products of self-similar semigroups which do retain self-similarity.
In \cite[Theorem~2]{bc_automaton1}, the first author and Cain showed that the presence of left identities in both factors is enough to guarantee a 
free product of two automaton semigroups to be an automaton semigroup, but conjectured that
there exist finite semigroups $S$ and $T$ such that $S\star T$ is not an
automaton semigroup \cite[Conjecture~5]{bc_automaton1}.  In \cite[Theorems~2 to 4]{bc_automaton2}, the same authors showed, however,
that not only is every free product of finite semigroups an automaton semigroup,
but whenever $S$ and $T$ are automaton semigroups either each containing an 
idempotent or both homogeneous (with respect to the presentation given by the generating automaton), then $S\star T$ is an automaton semigroup.
For her Bachelor thesis \cite{welker}, the third author modified the construction in \cite[Theorem~4]{bc_automaton2} to considerably relax the hypothesis on the base semigroups:
If $S$ and $T$ are automaton semigroups such that there exist automata for $S$ and $T$ with 
state sets $P$ and $Q$ respectively and maps $\phi:P\ra Q$ and $\psi: Q\ra P$
that extend to homomorphisms from $S$ to $T$ and from $T$ to $S$ respectively,
then $S\star T$ is an automaton semigroup \cite[Theorem~3.0.1]{welker}.
This is a strict generalization because, firstly, the maps
$\phi$ and $\psi$ can always be found if $S$ and $T$ either both contain idempotents 
(map all elements to a fixed idempotent in the other semigroup)
or are both homogeneous (map all elements to an arbitrary fixed element of the other semigroup); 
and, secondly, there is an example \cite[Example~3.3.3]{welker} of a free product satisfying 
the hypothesis of her theorem but not of \cite[Theorem~4]{bc_automaton2}.

In this paper, we extend the idea of the constructions used for these results in multiple directions.
First, we generally consider partial automata for all of our results, i.\,e.\ we do not require the generating automaton to be complete (in contrary to many other results in the literature, for example those mentioned above; see \cite{structurePart} for some results on the difference between using partial and complete automata for generating algebraic structures). Second, we show that the hypothesis on the base semigroups can be relaxed still further: If $S$ and $T$ are two automaton semigroups such that there exists a homomorphism
from one to the other, then their free product $S\star T$ is an automaton semigroup (\autoref{cor:autSmgrp}). This is again a strict generalization of \cite[Theorem~3.0.1]{welker} (even if we only consider complete automata).
Third, we show this result in the more general setting of self-similar semigroups\footnote{Note that the constructions from \cite[Theorem~2]{bc_automaton1}, \cite[Theorem~4]{bc_automaton2} and \cite{welker} mentioned above do not use that the generating automata for $S$ and for $T$ are finite. Therefore, these constructions also work for self-similar semigroups, although this is not explicitly stated there.} (\autoref{thm:main}) but observe that the constructed generating automaton for $S \star T$ is finite (and/or complete) if this was the case for the original two automata generating $S$ and $T$. The existence of a homomorphism from $S$ to $T$ (or vice-versa) is a very lax requirement and is satisfied by large classes of semigroups. For example, it suffices to have an idempotent (\autoref{cor:idempotent}) or a length function (\autoref{cor:lengthFunction}) in (at least) one of the two semigroups. By induction, we can even extend the result to arbitrary free products of (finitely many) semigroups where at least one contains an idempotent (\autoref{cor:ind}). The construction itself yields further results. As an example, we modify it to show that a new free generator can be adjoined to any self-similar semigroup (or automaton semigroup) without losing the property of self-similarity (or being an automaton semigroup; \autoref{thm:freeGenerator}). This is noteworthy because -- as mentioned above -- the free semigroup of rank one is not an automaton semigroup (not even if we allow partial automata, see \cite[Theorem~19]{structurePart} and \cite[Theorem~1.2.1.4]{waechter2020automaton}).

While our main result significantly relaxes the hypothesis for showing that the free product of self-similar semigroups (or automaton semigroups) is self-similar (an automaton semigroup), it does not settle the underlying question whether these semigroup classes are closed under free product. It is possible that there is a different construction for the free product $S \star T$ of two self-similar or automaton semigroup without the requirement of a homomorphism from one to the other and it is also possible that there is a pair of self-similar (or automaton) semigroups such that $S \star T$ is not a self-similar (or an automaton semigroup). In this case, however, no homomorphism $S \to T$ or $T \to S$ can exist. Thus, to make progress in either direction (towards a better construction or towards a counter-example), we need to look at pairs $S, T$ of self-similar (or even automaton) semigroups without a homomorphism from one to the other. However, it turns out that finding such a pair is not easy. In particular, neither $S$ nor $T$ may contain an idempotent. Thus, we have to consider idempotent-free semigroups here. We will show, however, that we cannot find a pair of such semigroups in the class of finitely generated simple semigroups. More precisely, using results by Jones on idempotent-free semigroups \cite{jones_bicyclic}, we show that finitely generated simple (or $0$-simple) idempotent-free semigroups are not residually finite (\autoref{thm:simpleIdempotentFreeIsNotResiduallyFinite}) and, thus, not self-similar (and, in particular, not automaton semigroups; \autoref{cor:simpleIdempotentIsNotAutomaton}). We then conclude the paper with an example\footnote{The authors would like to thank Emanuele Rodaro for his help in finding this example.} of a finitely generated residually finite semigroup (\autoref{prop:SnotResiduallyFinite}) which has no homomorphism to its opposite semigroup (\autoref{prop:SHasNoHoms}). While this comes close to the sought pair $S, T$, it is not clear whether the given semigroup is self-similar (\autoref{q:SSelfSimilar}).

\section{Preliminaries}

\paragraph*{\textbf{Functions, Alphabets and Words.}}
We write function applications on the right. For example, for a function $f: A \to B$ and an element $a \in A$, we write $af = (a)f$ for the (unique) value of $f$ at $a$.

An \emph{alphabet} is a finite, non-empty set $A$ and a finite sequence $w = a_1 \dots a_n$ of elements 
$a_1, \dots, a_n \in A$ is a \emph{string} or \emph{word} (over $A$), whose \emph{length} is $n$. 
The set of all strings over $A$, including the one of length $0$ (the \emph{empty string}, denoted $\varepsilon$), is denoted by $A^*$. 
Furthermore, we let $A^+ = A^* \setminus \{ \varepsilon \}$. 
We will also sometimes consider finite sequences over an infinite set. 

\paragraph*{\textbf{Semigroups and Their Free Products.}}
We assume familiarity with the basic concepts of semigroup theory and use common notation from this area.
A semigroup $S$ is \emph{generated} by a set $Q$ if every element $s \in S$ can be written as a product $q_1 \dots q_n$ of factors from $Q$.
If there exists a finite generating set for $S$, then $S$ is \emph{finitely generated}.
For a semigroup $S$ generated by $Q$, and two words $w,w'\in Q^+$, we write $w =_S w'$ if $w$ and $w'$ represent the same element of $S$.

Every semigroup $S$ can be presented as the set $Q^+ / \mathcal{R}$ of the classes of some congruence relation 
$\mathcal{R} \subseteq Q^+ \times Q^+$ with the (well-defined) operation $[w_1]_\mathcal{R} [w_2]_\mathcal{R} = [w_1 w_2]_\mathcal{R}$. 
We also use the standard presentation of semigroups $\langle Q \mid \ell_1 = r_1, \dots \rangle$ with $\ell_1, \dots, r_1, \ldots \in Q^+$, 
which denotes the semigroup $Q^+ / \mathcal{R}$ where $\mathcal{R} \subseteq Q^+ \times Q^+$ is the smallest congruence with 
$(\ell_1, r_1), \ldots \in \mathcal{R}$. 

The \emph{free product} of two semigroups $R = \langle P \mid \mathcal{R} \rangle$ and $S = \langle Q \mid \mathcal{S} \rangle$ 
(with $P \cap Q = \emptyset$) is the semigroup with presentation $\langle  P \cup Q \mid \mathcal{R} \cup \mathcal{S} \rangle$, denoted $R\star S$.

Note that there is a difference between the free product in the category of semigroups and the free product in the category of monoids or groups. 
In particular, in the semigroup free product (which we are exclusively concerned with in this paper) there is no amalgamation over the identity element of two monoids. Thus, the free product (in the category of semigroups) of two groups, for example, is not a group.

\paragraph*{\textbf{Automata.}}
In the setting of this work, an \emph{automaton} is a triple $\A= (Q, A, \delta)$ where 
\begin{itemize}
  \item $Q$ is a set, whose elements are called \emph{states}, 
  \item $A$ is a finite alphabet, whose elements are called \emph{letters} or \emph{symbols}, and 
  \item $\delta$ is a partial function $\delta: Q \times A \rightarrow Q \times A$, the \emph{transition function}.
\end{itemize}
In the theory of automaton semigroups, the definition of automata used is often more restrictive than this, with $Q$ required to be finite, 
and $\delta$ required to be a total function.  (Recall that the alphabet $A$ is, by definition, finite.)
So the reader should be aware that, in contrast to many other works, we will explicitly use the term \emph{finite automaton} when we require $Q$ to be finite and we explicitly call the automaton \emph{complete} when $\delta$ is a total function.

In more automata-theoretic settings, a finite automaton would be called a deterministic finite state, letter-to-letter (or synchronous) transducer (see for example \cite{lawson2004finite, linz2011introduction} for introductions on standard automata theory). However, the term \emph{automaton} is standard in our algebraic setting (although often only complete automata are considered).

A transition $(p,a)\delta = (q,b)$ in an automaton $\A$ is understood to mean that if $\A$ is in state $p$ and reads input $a$, 
then $\A$ moves to state $q$ and outputs $b$.  
In the typical graphical representation of automata, such a transition is depicted as
\begin{center}
  \begin{tikzpicture}[baseline=(q.base), auto, shorten >=1pt, >=latex]
    \node[state] (p) {$p$};
    \node[state, right=of p] (q) {$q$};
    \draw[->] (p) edge node {$a / b$} (q);
  \end{tikzpicture}
  .
\end{center}
The above transition can also be depicted as a \emph{cross diagram}:
\begin{center}
    \begin{tikzpicture}[>=latex, baseline=(m-2-3.base)]
      \matrix (m) [matrix of math nodes, text height=\ht\strutbox, text depth=0.25ex]
      {
          & a &   \\
        p &   & q \\
          & b &   \\
      };
      \drawArrows{2}{2}
    \end{tikzpicture}.
\end{center}
Multiple cross diagrams can be combined. For example, \autoref{sfig:multipleCrossDiagrams} combines the transitions
$(q_{i, j}, a_{i, j})\delta = (q_{i, j + 1}, a_{i + 1, j})$ for all $0 \leq i < n$ and $0 \leq j < m$ 
(where $q_{i, j} \in Q$ and $a_{i, j} \in A$). The same cross diagram can also be abbreviated using the 
short-hand notation given in \autoref{sfig:shortHandCrossDiagrams}, where we let 
$w = q_{1, 0} q_{2, 0} \dots q_{n, 0}$, $w' = q_{1, m} q_{2, m} \dots q_{n, m}$, $\alpha = a_{0, 1} a_{0, 2} \dots a_{0, m}$ and $\alpha' = a_{n, 1} a_{n, 2} \dots a_{n, m}$.

\begin{figure}\centering
  \begin{subfigure}{0.6\linewidth}\centering
    \begin{tikzpicture}[>=latex, auto]
      \matrix (m) [matrix of math nodes, text height=\ht\strutbox, text depth=0.25ex]
      {
                 & a_{0, 1} &          & a_{0, 2} &        & a_{0, m} &          \\
        q_{1, 0} &          & q_{1, 1} &          & \cdots &          & q_{1, m} \\
                 & a_{1, 1} &          & a_{1, 2} &        & a_{1, m} &          \\
        q_{2, 0} &          & q_{2, 1} &          & \cdots &          & q_{2, m} \\
                 & \vdots   &          & \vdots   &        & \vdots   & \\
        q_{n, 0} &          & q_{n, 1} &          & \cdots &          & q_{n, m} \\
                 & a_{n, 1} &          & a_{n, 2} &        & a_{n, m} &          \\
      };
      \drawArrows{2,4,6}{2,4,6}
      \draw[decorate, decoration={brace}] (m-6-1.south west) -- node {$w$} (m-2-1.north west);
      \draw[decorate, decoration={brace}] (m-1-2.north west) -- node {$\alpha$} (m-1-6.north east);
      \draw[decorate, decoration={brace}] (m-2-7.north east) -- node {$w'$} (m-6-7.south east);
      \draw[decorate, decoration={brace}] (m-7-6.south east) -- node {$\alpha'$} (m-7-2.south west);
    \end{tikzpicture}
    \caption{Multiple combined cross diagrams}\label{sfig:multipleCrossDiagrams}
  \end{subfigure}%
  \begin{subfigure}{0.4\linewidth}\centering
    \begin{tikzpicture}[>=latex, auto]
      \matrix (m) [matrix of math nodes, text height=\ht\strutbox, text depth=0.25ex]
      {
          & \alpha  &    \\
        w &         & w' \\
          & \alpha' &    \\
      };
      \drawArrows{2}{2}
    \end{tikzpicture}
    \caption{Short-hand notation}\label{sfig:shortHandCrossDiagrams}
  \end{subfigure}
  \caption{Combining cross diagrams}
\end{figure}

Intuitively, the rows of a cross diagram indicate successive runs of the automaton. 
For example, the $i^\textnormal{th}$ row of the cross diagram in \autoref{sfig:multipleCrossDiagrams} states that 
if we start in state $q_{i, 0}$ and read the input $a_{i - 1, 1} a_{i - 1, 2} \dots a_{i - 1, m}$, we obtain the output 
$a_{i, 1} a_{i, 2} \dots a_{i, m}$ and end in state $q_{i, m}$.

We refer to elements of $Q^+$ as \emph{words} and to elements of $A^+$ as \emph{strings}.
Since our automata $\A = (Q, A, \delta)$ are deterministic, 
for every word $w\in Q^+$ and string $\alpha\in A^+$
there is at most one $w' \in Q^+$ and 
$\alpha' \in A^+$ such that the cross diagram in 
\autoref{sfig:shortHandCrossDiagrams} holds. This way, $\A$ induces two actions:\footnote{Just like with automata, an \emph{action} in general may be partial.} one action of $Q^+$ on $A^+$ and 
one of $A^+$ on $Q^+$. The former is given by $\alpha \cdot w = \alpha'$ and the latter, the so-called \emph{dual action}, is given by $w @ \alpha = w'$. If there is no valid cross diagram with $w$ on the left and $\alpha$ at the top (i.\,e.\ if $w'$ and $\alpha'$ do not exist), then $\alpha \cdot w$ and $w @ \alpha$ are both undefined.
Furthermore, we extend the actions to $A^*$ by letting $\varepsilon \cdot w = \varepsilon$ and $w @ \varepsilon = w$ for all $w \in Q^+$ and to $Q^*$ by letting $\alpha \cdot \varepsilon = \alpha$ and $\varepsilon @ \alpha = \varepsilon$ for all $\alpha \in A^*$.
The string $\alpha' = \alpha\cdot w$ is the \emph{output} of $\A$ upon acting on $\alpha$ by $w$, while the word $w' = w @ \alpha$ is the 
\emph{restriction of $w$ at $\alpha$} (in other works sometimes denoted $w|_\alpha$).
There is a strong connection between the two actions defined by $\A$:
we have $\alpha \beta \cdot w = (\alpha \cdot w) (\beta \cdot [w @ \alpha])$ (or at least one term on either side is undefined).

Using the action $\alpha\mapsto \alpha \cdot w$, we can define the semigroup $\Sigma(\A)$ \emph{generated} by an automaton 
$\A = (Q, A, \delta)$. For this, we define the congruence $=_{\A}$ over $Q^*$ as
\[
  w =_{\A} w' \iff \forall \alpha \in A^*: \alpha \cdot w = \alpha \cdot w' \text{ (or both undefined).}
\]
While we define the congruence over $Q^*$, we are only interested in the generated semigroup and let $\Sigma(\A) = Q^+ / {=_{\A}}$ (i.\,e.\ we do not consider the empty word which would yield an identity element).
A semigroup arising in this way is called \emph{self-similar}. Furthermore, if the generating automaton is finite, it is an \emph{automaton semigroup}. 
If the generating automaton is additionally complete, we speak of a \emph{completely self-similar} semigroup or of a \emph{complete automaton semigroup}.

\begin{figure}\centering
  \begin{tikzpicture}[auto, shorten >=1pt, >=latex]
    \node[state] (+1) {$s$};
    \node[state, right=of +1] (+0) {$e$};
    
    \path[->] (+1) edge[loop left] node {$1/0$} (+1)
                   edge node {$0/1$} (+0)
              (+0) edge[loop right] node[align=left] {$0/0$\\$1/1$} (+0)
    ;
  \end{tikzpicture}
  \caption{The (binary) adding machine}\label{fig:addingMachine}
\end{figure}
\begin{figure}\centering
  \begin{tikzpicture}[auto, shorten >=1pt, >=latex]
    \node[state] (+1) {$s'$};
    \node[state, right=of +1] (+0) {$e$};
    
    \path[->] (+1) edge[loop left] node[fill=gray!25, inner sep=1pt, outer sep=1pt] {$1/1$} (+1)
                   edge node {$0/1$} (+0)
              (+0) edge[loop right] node[align=left] {$0/0$\\$1/1$} (+0)
    ;
  \end{tikzpicture}
  \caption{The unary adding machine}\label{fig:unaryAddingMachine}
\end{figure}
\begin{example}\label{ex:addingMachine}
  Consider the (complete and finite) automaton $\A = (\{ s, e \}, \{ 0, 1 \}, \delta)$ depicted in \autoref{fig:addingMachine}, called the \emph{adding machine}. 
Since we have $\alpha \cdot e = \alpha$ for all $\alpha \in \{ 0, 1 \}^*$, the state $e$ acts as the identity map and is, thus, a neutral element of $\Sigma(\A)$. In order to understand the action of $s$, we look at the cross diagram
  \begin{center}
    \begin{tikzpicture}[>=latex, baseline=(m-8-7.base)]
      \matrix (m) [matrix of math nodes, text height=\ht\strutbox, text depth=0.25ex]
      {
          & 0 &   & 0 &   & 0 &   \\
        s &   & e &   & e &   & e \\
          & 1 &   & 0 &   & 0 &   \\
        s &   & s &   & e &   & e \\
          & 0 &   & 1 &   & 0 &   \\
        s &   & e &   & e &   & e \\
          & 1 &   & 1 &   & 0 &   \\
        s &   & s &   & s &   & e \\
          & 0 &   & 0 &   & 1 &   \\
      };
      \drawArrows{2,4,6,8}{2,4,6}
    \end{tikzpicture}.
  \end{center}
  If we consider strings from $\{ 0, 1 \}^+$ as binary numbers in reverse (with the least significant bit on the left), it is easy to 
observe that $s$ acts as an increment on these strings (although an overflow might occur). 
Thus, $\Sigma(\A)$ is isomorphic to the free monoid in one generator, which is hence a complete automaton semigroup.
\end{example}

\begin{example}\label{ex:unaryAddingMachine}
  The free monoid in one generator can also be presented in a slightly different way as an automaton semigroup. For this, we can modify the adding machine from \autoref{ex:addingMachine} to use a unary instead of a binary counter. Let $\A' = (\{ s', e \}, \{ 0, 1 \}, \delta')$ be the automaton depicted in \autoref{fig:unaryAddingMachine}. Note that the only change compared to the adding machine from \autoref{ex:addingMachine} and \autoref{fig:addingMachine} is that we now have a $1/1$-self loop at $s'$.
  
  The state $e$ still acts as the identity map but, regarding $s'$, we now have
  \begin{center}
    \begin{tikzpicture}[>=latex, baseline=(m-6-7.base)]
      \matrix (m) [matrix of math nodes, text height=\ht\strutbox, text depth=0.25ex]
      {
           & 0 &    & 0 &    & 0 &   \\
        s' &   & e  &   & e  &   & e \\
           & 1 &    & 0 &    & 0 &   \\
        s' &   & s' &   & e  &   & e \\
           & 1 &    & 1 &    & 0 &   \\
        s' &   & s' &   & s' &   & e \\
           & 1 &    & 1 &    & 1 &   \\
      };
      \drawArrows{2,4,6}{2,4,6}
    \end{tikzpicture}.
  \end{center}
  or, more generally, $1^k 0^\ell \cdot (s')^n = 1^{k + n} 0^{\ell - n}$ and $(s')^n @ 1^k 0^\ell = e^n$ (for $\ell \geq n$), which again shows that all powers of $s'$ are distinct in the semigroup.
\end{example}

\begin{figure}[b]\centering
  \begin{tikzpicture}[auto, shorten >=1pt, >=latex]
    \node[state] (a) {$a$};
    \node[state, right=of a] (b) {$b$};
    \draw[->] (a) edge[loop left] node {$a/a$} (a)
                  edge[bend left] node {$b/a$} (b)
              (b) edge[loop right] node {$b/b$} (b)
                  edge[bend left] node {$a/b$} (a);
  \end{tikzpicture}
  \caption{An automaton generating the free semigroup $\{ a, b \}^+$}\label{fig:abplus}
\end{figure}
\begin{example}\label{ex:freeSemigroup}
  Let $\A= (Q, A, \delta)$ be the complete and finite automaton with state set $Q=\{a,b\}$, alphabet $\Sigma=\{a,b\}$ and the transitions
  \begin{align*}
    \delta: Q \times A &\to Q \times A \\
    (c, d) &\mapsto (d, c), \quad c,d\in \{a,b\}
  \end{align*}
  (see \autoref{fig:abplus} and \cite[Proposition~4.1]{cain_1auto}).
  
  The action of $w = bab \in Q^+$ on $\alpha = aaa$ is given by the cross diagram
  \begin{center}
    \begin{tikzpicture}[>=latex, baseline=(m-6-7.base)]
      \matrix (m) [matrix of math nodes, text height=\ht\strutbox, text depth=0.25ex]
      {
          & a &   & a &   & a &   \\
        b &   & a &   & a &   & a \\
          & b &   & a &   & a &   \\
        a &   & b &   & a &   & a \\
          & a &   & b &   & a &   \\
        b &   & a &   & b &   & a \\
          & b &   & a &   & b &   \\
      };
      \drawArrows{2,4,6}{2,4,6}
    \end{tikzpicture}.
  \end{center}
  Thus, we have $aaa \cdot bab = bab$ and $bab @ aaa = aaa$.
  
  It is not difficult to see that we have $\alpha \cdot w = w$ and $w @ \alpha = \alpha$, in general,
for all $w \in Q^n$ and $\alpha \in A^n$ (for the same $n > 0$). Thus, all pairs of distinct words $w, w' \in Q^+$ 
act differently on some witness and we obtain $\Sigma(\A) = \{ a, b \}^+$.
\end{example}

\paragraph*{\textbf{Assuming Non-Trivial Action.}}
Any self-similar semigroup can be generated by an automaton in which no word acts trivially (i.\,e.\ it acts as the identity map) on the set of all strings.
Let $\A = (Q,A,\delta)$ and for an arbitrary $a\in A$, let $A'=A\cup \{a'\}$ where $a'$ is a symbol not already in $A$.
Define a new automaton $\A' = (Q, A', \delta')$
with $\delta'$ being the extension of $\delta$ to $Q\times A'$ given by $(q,a')\delta' = (q,a)\delta$ (or both undefined) for all $q\in Q$.
Then $\Sigma(\A')\cong \Sigma(\A)$ and no word in $Q^+$ acts trivially on $(A')^*$. Furthermore, $\A'$ is complete if $\A$ is. We summarize this in the following fact.
\begin{fact}\label{fct:nonTrivialAction}
  For every (complete) automaton $\A$ with state set $Q$, there is a (complete) automaton $\A'$ with state set $Q$, $\Sigma(\A')\cong \Sigma(\A)$ and where no word $w \in Q^+$ acts as the 
identity map (i.\,e.\ for every $w \in Q^+$, there is some string $\alpha$ over the alphabet of $\A'$ such that $\alpha \cdot w \neq \alpha$, including the case that the left-hand side is undefined).
\end{fact}
\noindent{}Note that the constructed automaton $\A'$ remains finite if $\A$ is finite and that the construction is computable (under suitable computability assumptions on $\A$).

\paragraph*{\textbf{Union and Power Automaton.}}
From two automata $\A = (P, A, \sigma)$ and $\B = (Q, A, \tau)$ over a common alphabet (with disjoint state sets), we can build their \emph{union automaton} $\A \cup \B = (P \cup Q, A, \sigma \cup \tau)$ whose transitions are given by
\begin{align*}
  \sigma \cup \tau : (P \cup Q) \times A &\to (P \cup Q) \times A\\
  (r, a) &\mapsto
    \begin{cases}
      \sigma(r, a) & \text{if } r \in P, \\
      \tau(r, a) & \text{if } r \in Q \text{.}
    \end{cases}
\end{align*}
The union of two finite (complete) automata is finite (complete).

From the same two automata, we can also build their \emph{composition automaton} $\A \B = (PQ, A, \sigma\tau)$ where $PQ = \{ pq \mid p \in P, q \in Q \}$ is the cartesian product of $P$ and $Q$ and the transitions are given by $\sigma\tau : PQ \times A \to PQ \times A$ where $(pq, a) \mapsto (p' q', c)$ when $(p, a)\sigma = (p', b)$ is defined and $(q, b)\tau = (q', c)$ is also defined for $b = a\cdot p$ (where the action is with respect to $\A$). Otherwise, $\sigma\tau$ is undefined at $(pq, a)$. Note that the composition of automata is associative and that the resulting automaton is finite (complete) if both original automata were.

The \emph{$i^{\textit{th}}$ power automaton} $\A^i = (Q^i, A, \delta^i)$ of an automaton $\A = (Q, A, \delta)$ is the $i$-fold composition of $\A$ with itself. Note that any $w \in Q^i$ can be seen as a word over the states of $\A$ and as a state of $\A^i$. However, the corresponding actions of $w$ on some $\alpha \in A^*$ under both views coincide. This allows us to consider any word $w \in Q^+$ as a state of the automaton $\A \cup \A^{|w|}$ without changing the generated semigroup (i.\,e.\ we have $\Sigma(\A \cup \A^{|w|}) = \Sigma(\A)$). 
This is a typical application of the power automation.

\begin{fact}\label{fact:power}
  For any automaton $\mathcal{A} = (Q, A, \delta)$ and any finite subset $X \subseteq \Sigma(\mathcal{A})$, there exists an automaton $\mathcal{A}' = (Q', A, \delta')$ generating $\Sigma(\mathcal{A})$ with $X \subseteq Q'$. This automaton is finite (complete) if $\mathcal{A}$ is.
\end{fact}

Note that, since the composition and, thus, the power construction is computable, the automaton from \autoref{fact:power} can be computed on the input of an automaton generating $S$ and the finite set $X$ as words over the state set of that automaton.

\section{The main theorem}

We prove our main theorem using a construction heavily inspired by the original 
construction in \cite[Theorem~4]{bc_automaton2} and the modifications for the generalization in \cite{welker}.
The major difference is that the earlier constructions were symmetric in the base semigroups $S$ and $T$.
Although both works refer only to automaton semigroups, the finiteness of the generating set is never used
in either, and the constructions therein work also for general (but complete) self-similar semigroups.

\begin{theorem}\label{thm:main}
Let $S$ and $T$ be (completely) self-similar semigroups such that there exists
a homomorphism from $S$ to $T$.
Then $S\star T$ is a (completely) self-similar semigroup.
\end{theorem}
\begin{proof}
Let $\A = (P, A, \sigma)$ and $\B = (Q, B, \tau)$ be generating automata for $S$ and $T$ 
respectively with $P \cap Q = \emptyset$ and $A \cap B = \emptyset$.
Let $\phi: P^+ \to Q^+$ be a map which induces a homomorphism $S \to T$.
By Fact~\ref{fact:power}, we may assume that $s\phi\in Q$ for all $s\in P$ and extend $\phi$ into a homomorphism 
$(P \cup Q)^* \to Q^*$ by letting $\varepsilon \phi = \varepsilon$ and $t \phi = t$ for all $t \in Q$.
Note that, thus extended, $\phi$ also induces a homomorphism $S \star T \to T$.

By Fact~\ref{fct:nonTrivialAction}, we can, without further modifying the state sets of the generating automata,
assume that no word in $P^+$ (respectively $Q^+$) acts as the identity map on $A^*$ (respectively $B^*$).
Note that this modification does not disrupt our assumption that $s\phi\in Q$ for all $s\in P$.

We construct an automaton $\C = (R, C, \delta)$ with $\Sigma(\C) = S \star T$.
Let $R = P \cup Q$ and
\[
  C = \gset{ a, a^S, a^\circ } {a \in A} \cup \gset{ b, b^S, b^T, b^\circ } {b \in B} \cup \{ \$, \hat{\$}, \ol{\$} \} \text{.}
\]
For $a \in A$ and $b \in B$, we say that $a$ and $b$ are \emph{unmarked}, $a^S$ and $b^S$ are 
\emph{$S$-marked}, $b^T$ is \emph{$T$-marked} and $a^\circ$ and $b^\circ$ are \emph{circled}. 
We use this notation also for strings. For example, for $\alpha = a_1 \dots a_\ell$ with 
$a_1, \dots, a_\ell \in A$, we write $\alpha^\circ$ for the string $a_1^\circ \dots a_\ell^\circ$.
The symbols $\$, \hat{\$}, \ol{\$}$ are called \emph{gates} and we refer to them respectively as 
\emph{open}, \emph{half-open} and \emph{closed} gates.
The transformation $\delta$ is defined as follows.
For $s\in P$, $t\in Q$, $a\in A$, $b\in B$ with $(s,a)\sigma = (s_0,a_0)$ (if defined) and $(t,b)\tau = (t_0,b_0)$ (if defined),
the action of $R$ on $R \times C$ is given by (see also \autoref{fig:schematic}):%
\renewcommand{\arraystretch}{1.5}%
\begin{center}%
  \resizebox{\linewidth}{!}{%
  \begin{tabular}{r|cccccccccc}
        & $a$ & $a^S$ & $a^\circ$ & $b$ & $b^S$ & $b^T$ & $b^\circ$ & $\$$ & $\hat{\$}$ & $\ol{\$}$ \\\hline
    $s$ & $(s_0, a_0^S)$ & $(s_0, a_0^S)$ & $(s, a^\circ)$ & $(s, b^S)$ & $(s, b^S)$ & $(s, b^\circ)$ & $(s, b^\circ)$ & $(s, \$)$ & $(s\phi, \hat{\$})$ & $(s\phi, \hat{\$})$ \\
    $t$ & $(t, a)$ & $(t, a^\circ)$ & $(t, a^\circ)$ & $(t, b)$ & $(t_0, b_0^T)$ & $(t_0, b_0^T)$ & $(t, b^\circ)$ & $(t, \$)$ & $(t, {\$})$ & $(t, \ol{\$})$
  \end{tabular}}
\end{center}
If $(s,a)\sigma$ or $(t,b)\tau$ is undefined, then $\delta$ is also undefined on the corresponding inputs. Note that $\C$ is complete if (and only if) $\A$ and $\B$ are.

\begin{figure}\centering
  \begin{tikzpicture}[auto, shorten >=1pt, >=latex]
    \node[state] (s) {$s$};
    \node[state, right=2cm of s, dashed] (s_0) {$s_0$};
    
    \draw[->] (s) edge[dotted] node[align=center]
                    {$a^{\phantom{S}} / a_0^S$\\
                     $a^{S} / a_0^S$} (s_0)
                  edge[loop below] node[align=center]
                    {$a^\circ / a^\circ$\\
                     $b^\circ / b^\circ$} (s)
                  edge[loop left] node[align=center] (1left)
                    {$b^{\phantom{T}} / b^S$\\
                     $b^{\makebox[0pt][l]{$\scriptstyle S$}\phantom{T}} / b^S$\\
                     $b^{T} / b^\circ$} (s)
                  edge[loop above] node (1top) {$\$/\$$} (s)
    ;

    \node[state, right=3cm of s_0] (t) {$t$};
    \node[state, right=2cm of t, dashed] (t_0) {$t_0$};
    
    \draw[->] (t) edge[dotted] node[align=center]
                    {${b}^{S} /\, b_0^T$\\
                     ${b}^{T} /\, b_0^T$} (t_0)
                  edge[loop below] node[align=center]
                    {$a / a$\\
                     $b / b$\\
                     $a^\circ / a^\circ$\\
                     $b^\circ / b^\circ$} (t)
                  edge[loop left] node[align=center] (2left)
                    {$a^S / a^\circ$} (t)
                  edge[loop above] node[align=center] (2top) {$\ol{\$}/\ol{\$}$\\$\hat{\$}/{\$}$\\$\$/\$$} (t)
    ;
    
    \node[state, below=3cm of t, dashed] (sphi) {$s\phi$};
    
    \draw[->] (s) edge node[swap, align=center] {$\ol{\$}/\hat{\$}$\\$\hat{\$}/\hat{\$}$} (sphi);
  
    \path (2left.west) |- coordinate (2topleft) (2top.north);
    \path (sphi.south) -| coordinate (2bottomright) (t_0.east);
    \draw[rounded corners, dashed] ($(2topleft)+(-0.25cm,0.5cm)$) rectangle ($(2bottomright)+(0.25cm,-0.5cm)$);
    \path ($(2topleft)+(-0.25cm,0.5cm)$) -| node[below left] {$\B$} ($(2bottomright)+(0.25cm,-0.5cm)$);
    
    \path (1left.west) |- coordinate (1topleft) (2topleft);
    \path (s_0.east) |- coordinate (1bottomright) (sphi.south);
    \draw[rounded corners, dashed] ($(1topleft)+(-0.25cm,0.5cm)$) rectangle ($(1bottomright)+(0.25cm,-0.5cm)$);
    \node[anchor=north west] at ($(1topleft)+(-0.25cm,0.5cm)$) {$\A$};
  \end{tikzpicture}
  \caption{Schematic representation of $\C$. The dotted transitions only exist if the corresponding transitions exist in $\A$ or $\B$, respectively.}\label{fig:schematic}
\end{figure}

The general idea of the construction is that the action of a word $w = u_1 v_1 \dots u_n v_n$ 
with $u_1, \dots, u_n \in P^+$ and $v_1, \dots, v_n \in Q^+$ on a string
\[
  \gamma = {\alpha_1}{\beta_1} \, \ol\$ \, {\alpha_2}{\beta_2} \, \ol\$ \dots \ol\$ \, {\alpha_n}{\beta_n},
\] 
with $\alpha_i\in A^*, \beta_i\in B^*$ is defined if and only if each $\alpha_i\cdot u_i$ and $\beta_i\cdot v_i$ for $1\leq i\leq n$ is defined, in which case it is given by
\[
  \gamma \cdot w = {(\alpha_1 \cdot u_1)}^\circ {(\beta_1\cdot v_1)}^\circ \, \$ \ldots \$ \, {(\alpha_{n-1} \cdot u_{n-1})}^\circ {(\beta_{n-1} \cdot v_{n-1})}^\circ \, 
\$ \ \, {(\alpha_n \cdot u_n)}^\circ {(\beta_n \cdot v_n)}^T \text{,}
\]
where the actions $\alpha_i \cdot u_i$  and $\beta_i \cdot v_i$ refer to those induced by $\A$ and $\B$,
respectively.

In order to formally show that $\Sigma(\C) = S \star T$, we prove
\[
  w =_{S \star T} \tilde{w} \iff \forall \gamma \in C^*: \gamma \cdot w = \gamma \cdot \tilde{w} \text{ (or both undefined)} \]
for all $w, \tilde{w} \in (P \cup Q)^+$.
That is, we prove that $\C$ defines an action of
$S\star T$ on $C^*$ (left to right implication) and that this action is faithful (right to left implication).

We begin with the direction from left to right (i.\,e.\ that $\C$ defines an action of $S \star T$ on $C^*$). It suffices to show the statement for all $w, \tilde{w} \in P^+$ with $w =_{\A} \tilde{w}$ (i.\,e.\ that the action of $P^+$ is an action of $\Sigma(\A) = S$) and for all $w, \tilde{w} \in Q^+$ with $w =_{\B} \tilde{w}$ (i.\,e.\ the action of $Q^+$ is an action of $\Sigma(\B) = T$) since this includes all relations in $S$ and in $T$ and, thus, all relations in $S \star T$.

We show that for $w,\tilde{w}\in P^+\cup Q^+$, if $w =_{\A \cup \B} \tilde{w}$ then for any $c\in C$ we have the implication
\begin{center}
  \begin{tikzpicture}[>=latex]
    \matrix (m) [matrix of math nodes,
      text height=\ht\strutbox, text depth=0.25ex,] {
        & c  & \\
      w &    & w' \\
        & c' & \\
    };
    \drawArrows{2}{2}

    \matrix[right=of m] (mtilde) [matrix of math nodes,
      text height=\ht\strutbox, text depth=0.25ex,] {
                & c  & \\
      \tilde{w} &    & \tilde{w}' \textcolor{gray}{{}=_{\A \cup \B} w'} \\
                & c' & \\
    };
    \drawArrows[mtilde]{2}{2}

    \path (m-2-3.base east) -- node[anchor=base] {$\implies$} (mtilde-2-1.base west);
  \end{tikzpicture}
\end{center}
where $w'$ (and, thus, $\tilde{w}'$) is in $P^+ \cup Q^+$.
At the same time, we also show that $c \cdot w$ is undefined if and only if $c \cdot \tilde{w}$ is.
Then, by induction, we obtain $\gamma \cdot w = \gamma \cdot \tilde{w}$ (or both undefined) for all $\gamma \in C^*$.

Since $t @ c\in Q$ for all $t \in Q$ and $c \in C$ (if it is defined), it is easier to first show the implication for $w, \tilde{w}\in Q^+$ with $w =_{\B} \tilde{w}$.
We say that a state $r \in R$ \emph{ignores} a symbol $c \in C$ if we have $(r, c)\delta = (r, c)$ (i.\,e.\ if there is a $c/c$-labeled self-loop at $r$).
With this terminology, all states in $Q$ ignore open gates, closed gates, and unmarked and circled letters, so the inductive hypothesis holds trivially for these (in particular, $c \cdot w$ and $c \cdot \tilde{w}$ are always defined in these cases).
\enlargethispage*{\baselineskip}
For the remaining types of symbols in $C$ we have the following cross-diagrams, and analogous ones for $\tilde{w}$:
\begin{center}
  \begin{tikzpicture}[>=latex]
    \matrix (m) [matrix of math nodes,
      text height=\ht\strutbox, text depth=0.25ex,] {
        & \hat{\$}  & \\
      w &    & w \\
        & \$ & \\
    };
    \drawArrows{2}{2}
    
    \matrix[right=of m] (m2) [matrix of math nodes,
      text height=\ht\strutbox, text depth=0.25ex,] {
        & a^S & \\
      w &     & w \\
        & a^\circ & \\
    };
    \drawArrows[m2]{2}{2}

    \matrix[right=of m2] (abS) [matrix of math nodes,
      text height=\ht\strutbox, text depth=0.25ex,] {
        & {b}^{{\scriptsize S}} \text{ or } {b}^{{\scriptsize T}} & \\
      w &    & w @ b \\
        & {(b\cdot w)}^T & \\
    };
    \drawArrows[abS]{2}{2},
  \end{tikzpicture}
\end{center}
The left diagram shows the implication for $\hat{\$}$ and the middle one for $a^S$. In the right one, 
$b\cdot w$ and $w @ b$ are as in $\B$, and the letter at the top may be either $S$-marked or $T$-marked. 
Note that this diagram does not exist if $b \cdot w$ is undefined in $\B$. However, this is the case if and only if $b \cdot \tilde{w}$ is undefined in $\B$ (as we have $w =_{\B} \tilde{w}$) and we then have that $b^S \cdot w$, $b^T \cdot w$, $b^S \cdot \tilde{w}$ and $b^T \cdot \tilde{w}$ are all undefined in $\C$.
Thus the induction hypothesis for $w, \tilde{w}\in Q^+$ follows since we have $b \cdot w = b \cdot \tilde{w}$ and $w @ b =_{\B} \tilde{w} @ b$ (if defined) as $w =_{\B} \tilde{w}$.

For $w, \tilde{w} \in P^+$ with $w =_{\A} \tilde{w}$, observe that all states in $P$ ignore open gates and circled letters, 
so the induction hypothesis holds trivially for these symbols.
For the remaining types of symbols, the cross diagrams are as follows:
\begin{center}
 \begin{tikzpicture}[>=latex]
    \matrix (m) [matrix of math nodes,
      text height=\ht\strutbox, text depth=0.25ex,] {
        & \hat{\$} \text{ or } \ol{\$}  & \\
      w &    & w\phi \in Q^+ \\
        & \hat{\$} & \\
    };
    \drawArrows{2}{2}
    
    \matrix[right=0.25cm of m] (aS) [matrix of math nodes,
      text height=\ht\strutbox, text depth=0.25ex,] {
        & {a} \text{ or } a^S  & \\
      w &    & w @ a \in P^+ \\
        & {(a \cdot w)}^S & \\
    };
    \drawArrows[aS]{2}{2}
    
    \matrix[right=0.25cm of aS] (bS) [matrix of math nodes,
      text height=\ht\strutbox, text depth=0.25ex,] {
        & {b} \text{ or } b^S  & \\
      w &    & w \\
        & {b}^S & \\
    };
    \drawArrows[bS]{2}{2}

    \matrix[right=0.25cm of bS] (bT) [matrix of math nodes,
      text height=\ht\strutbox, text depth=0.25ex,] {
        & {b}^T  & \\
      w &    & w  \\
        & {b}^\circ & \\
    };
    \drawArrows[bT]{2}{2}

  \end{tikzpicture}
\end{center}
The first diagram settles the remaining gates. Here, we have used that $\phi$ is a homomorphism (and that we have already covered the case $w, \tilde{w} \in Q^+$ above). In the second diagram, the letter may either be unmarked or $S$-marked and $a\cdot w$ and $w @ a$ are as in $\A$.
Again, this cross diagram does not exist if $a \cdot w$ is undefined in $\A$. However, then $a \cdot \tilde{w}$ must also be undefined (in $\A$) and the same is true for $a \cdot w$, $a^S \cdot w$, $a \cdot \tilde{w}$ and $a^S \cdot \tilde{w}$ in $\C$.
It is important to note here that we have $w@a, \tilde{w}@a \in P^+$ (if defined)
as for $s \in P, c \in C$, we have $s @ c\in Q$ only if $c$ is a closed or half-open gate,
and these types of gates never appear in the output for inputs from $A \cup A^S$. This settles the letters $a$
and $a^S$. The third diagram (where $b$ may be unmarked or $S$-marked) and the last diagram settle the remaining cases, which completes the proof of the left to right implication by induction.

It remains to show the implication from right to left (i.\,e.\ that the action of $S \star T$ on $C^*$ defined by $\C$ is faithful). We will prove this by contraposition, but first claim
\begin{equation}
  \begin{tikzpicture}[>=latex, baseline=(m-2-3.base)]
    \matrix (m) [matrix of math nodes,
      text height=\ht\strutbox, text depth=0.25ex,] {
      & \ol{\$}^k & \\
      u_1 v_1 \dots u_k v_k & & (u_1 v_1 \dots u_k v_k) \phi \textcolor{gray}{{}\in Q^+}\\
      & \$^k & \\
    };
    \drawArrows{2}{2}
  \end{tikzpicture}\label{eqn:mainTheoremClaim}
\end{equation}
for all $u_1, \dots, u_k \in P^+$, $v_1, \dots, v_k \in Q^+$ and $k \geq 0$. This follows by induction on $k$: for $k = 0$, there is nothing to show, and the inductive step is depicted in \autoref{fig:claimInductiveStep}.
\begin{figure}\centering
  \begin{tikzpicture}[>=latex]
    \matrix (m) [matrix of math nodes,
      text height=\ht\strutbox, text depth=0.25ex,] {
        & \ol{\$}^{k-1} & & \ol\$ \\
      u_1 v_1 \dots u_{k-1} v_{k-1} & & (u_1 v_1 \dots u_{k-1} v_{k-1}) \phi & & (u_1 v_1 \dots u_{k-1} v_{k-1}) \phi\\
        & \$^{k-1} & & \ol\$ \\
      u_k & & u_k & & u_k \phi \\
        & \$^{k - 1} & & \hat{\$} \\
      v_k & & v_k & & v_k = v_k \phi \\
        & \$^{k - 1} & & {\$} \\
    };
    \drawArrows{2,4,6}{2,4}
    \draw[dashed, rounded corners] (m-1-2.north -| m-2-1.west) rectangle (m-3-2.south -| m-2-3.east);
  \end{tikzpicture}
  \caption{Inductive step for the claim (\ref{eqn:mainTheoremClaim}). The dashed part follows by induction.}\label{fig:claimInductiveStep}
\end{figure}

Next, let $w, \tilde{w} \in (P \cup Q)^+$ represent different elements in $S \star T$ (i.\,e.\ $w \neq_{S \star T} \tilde{w}$). We may factorize $w$ into blocks and write
\begin{align*}
  w &= v_0 \, u_1 v_1 \dots u_m v_m \, u_{m + 1}
\end{align*}
for $m \geq 0$, $v_0 \in Q^*$, $u_1, \dots, u_m \in P^+$, $v_1, \dots, v_m \in Q^+$ and $u_{m + 1} \in P^*$.
Factorizing $\tilde{w}$ analogously yields $\tilde{m}$, $\tilde{u}_i$ and $\tilde{v}_i$ and 
we may assume $\tilde{m} \leq m$ without loss of generality. By setting $\tilde{u}_i = \varepsilon$ and 
$\tilde{v}_i = \varepsilon$ appropriately, we may write $\tilde{w} = \tilde{v}_0 \, \tilde{u}_1 \tilde{v}_1 \dots \tilde{u}_n \tilde{v}_m \, \tilde{u}_{m + 1}$.

Now, let $k$ be minimal such that $u_k\neq_{\A} \tilde{u}_k$ (i.\,e.\ there is a difference in an $S$-block) or $v_k\neq_{\B} \tilde{v}_k$ (i.\,e.\ there is a difference in a $T$-block). This includes the case that one side is empty.
If $k=0$, then $w$ and $\tilde{w}$ can already be distinguished by their actions
on a string ${\beta}^T$, where $\beta \in B^+$ is chosen such that $\beta \cdot v_0 \neq \beta\cdot \tilde{v}_0$ (or exactly one side is undefined). Such $\beta$ exists even if $v_0$ or $\tilde{v}_0$ is empty, since we chose $\B$ such that no word in $P^+$ acts as the identity map on $B^+$. We have
\begin{center}
  \begin{tikzpicture}[>=latex]
    \matrix (m) [matrix of math nodes,
      text height=\ht\strutbox, text depth=0.25ex,] {
          & {\beta}^T & \\
      v_0 & & v_0 @ \beta \\
          & {(\beta \cdot v_0)}^T & \\
      u_1 v_1 \dots u_m v_m \, u_{m + 1} & & u_1 v_1 \dots u_m v_m \, u_{m + 1} \\
          & {(\beta \cdot v_0)}^{\circ} & \\
      };
      \drawArrows{2,4}{2}
      \path[fill=gray, opacity=0.2, rounded corners] (m-4-1.west |- m-4-1.north) rectangle (m-4-3.east |- m-5-2.south);
  \end{tikzpicture}
\end{center}
if $\beta \cdot v_0$ is defined,
where the shaded part only exists for $u_1 v_1 \dots u_m v_m \, u_{m + 1} \neq \varepsilon$ (but is always defined in this case). Analogously, we get ${\beta}^T \cdot \tilde{w} \in \{ {(\beta \cdot \tilde{v}_0)}^T, {(\beta \cdot \tilde{v}_0)}^{\circ} \}$ (or that $\beta^T \cdot \tilde{w}$ is undefined). Either way, this shows $\beta^T \cdot w \neq \beta^T \cdot \tilde{w}$ (including the case that one side is undefined while the other one is not).

For $k > 0$, we now aim to find a string $\gamma \in C^*$ such that $ \gamma \cdot w \neq \gamma \cdot \tilde{w}$ for the action of $\C$.
We will choose $\gamma = \ol{\$}^{k-1} \alpha \beta$ for some $\alpha\in A^*, \beta\in B^*$.
If $u_k \neq_{\A} \tilde{u}_k$ (including the case that one side is empty, which may also happen for $u_{m + 1}$),
choose $\alpha \in A^+$ with $\alpha \cdot u_k \neq \alpha \cdot \tilde{u}_k$ in $\A$
(including the case that one side is defined while the other is not), and let $\beta = \varepsilon$.
Otherwise, $v_k \neq_{\B} \tilde{v}_k$ (including the case that $\tilde{v}_k$ is empty),
and we let $\alpha = \varepsilon$ and choose $\beta \in B^+$ with 
$\beta \cdot v_k \neq \beta \cdot \tilde{v}_k$ in $\B$ (including the case that one is defined while the other is not).
Note that in both cases $\alpha$ exists even if $u_k$ or $\tilde{u}_k$ is empty (while the other one is not), and that $\beta$ also exists if $\tilde{v}_k = \varepsilon$, since no word acts as the identity map in $\A$ or $\B$.
Then when $\alpha\cdot u_k$ and $\beta \cdot v_k$ are defined (in $\A$ and $\B$, respectively) we have 
$\gamma\cdot w = {\$}^{k - 1} (\alpha \cdot u_k)^{D_\alpha} (\beta \cdot v_k)^{D_\beta}$ where $D_\alpha, D_\beta \in \{ \varepsilon, S, T, \circ \}$.
This is shown by the cross diagram (of $\C$) depicted in \autoref{fig:caseKgt0}, where the 
light gray part does not exist if $u_k = \varepsilon$ (which may only happen if $k = m + 1$), the medium gray part
does not exist if $v_k = \varepsilon$, and the dark gray part does not exist if $u_{k + 1} = \varepsilon$. 

Since we have chosen $k$ to be minimal, we have $u_i =_{\A} \tilde{u}_i$ and $v_i =_{\B} \tilde{v}_i$ for $1 \leq i < k$ and, in particular, $\tilde{u}_i$ and $\tilde{v}_i$ are non-empty
for $1\leq i <k$. Therefore, we can still apply the claim (\ref{eqn:mainTheoremClaim}) and obtain an analogous cross diagram for $\tilde{w}$ if $\alpha \cdot \tilde{u}_k$ and $\beta \cdot \tilde{v}_k$ are both defined (in $\A$ and $\B$, respectively). Thus, in this case, $\gamma\cdot \tilde{w} = \$^{k - 1} (\alpha \cdot \tilde{u}_k)^{\tilde{D}_\alpha} (\beta \cdot \tilde{v}_k)^{\tilde{D}_\beta}$ 
(again: $\tilde{D}_\alpha, \tilde{D}_\beta \in \{ \varepsilon, S, T, \circ \}$) and,
by choice of $\alpha$ and $\beta$, we have $\gamma\cdot w \neq \gamma\cdot \tilde{w}$ in $\C$.

If $\alpha \cdot u_k$ is undefined while $\alpha \cdot \tilde{u}_k$ is defined (both in $\A$), we have $\beta = \varepsilon$ and obtain by the same cross diagram in \autoref{fig:caseKgt0} (and its analogue for $\tilde{w}$) that $\gamma \cdot w = {\$}^{k - 1} (\alpha \cdot u_k)^{D_\alpha} (\beta \cdot v_k)^{D_\beta}$ while $\gamma \cdot \tilde{w}$ is undefined (both in $\C$). The case that $\alpha \cdot u_k$ is defined while $\alpha \cdot \tilde{u}_k$ is not is symmetric. 
Finally, $\alpha \cdot u_k$ and $\alpha \cdot \tilde{u}_k$ cannot both be undefined, since our choice of $\alpha$ and $\beta$ would then imply that $\alpha = \varepsilon$, but
$\varepsilon\cdot u_k$ and $\varepsilon\cdot \tilde{u}_k$ are always defined (as $\varepsilon$).
The same arguments hold if $\beta \cdot v_k$ or $\beta \cdot \tilde{v}_k$ is undefined (in $\B$).

Thus, we have found the required string $\gamma$ in all cases and
$\C$ defines a faithful action of $S\star T$ on $C^*$,
which shows that $S\star T = \Sigma(\C)$ is self-similar, and completely self-similar if $S$ and $T$ both are.
\end{proof}

\begin{sidewaysfigure}\centering
  \begin{tikzpicture}[>=latex]
    \matrix (m) [matrix of math nodes,
      text height=\ht\strutbox, text depth=0.25ex,] {
        & \ol{\$}^{k - 1} & & \alpha & & \beta & \\
      v_0 & & v_0 & & v_0 & & v_0 \\
        & \ol{\$}^{k - 1} & & \alpha & & \beta & \\
      u_1 v_1 \dots u_{k - 1} v_{k - 1} & & (u_1 v_1 \dots u_{k - 1} v_{k - 1}) \phi & & (u_1 v_1 \dots u_{k - 1} v_{k - 1}) \phi & & (u_1 v_1 \dots u_{k - 1} v_{k - 1}) \phi\\
        & \$^{k - 1} & & \alpha & & \beta & \\
      u_k & & u_k & & u_k @ \alpha & & u_k @ \alpha \\
        & \$^{k - 1} & & (\alpha \cdot u_k)^S & & \beta^S & \\
      v_k & & v_k & & v_k & & v_k @ \beta \\
        & \$^{k - 1} & & (\alpha \cdot u_k)^\circ & & (\beta \cdot v_k)^T & \\
      u_{k + 1} v_{k + 1} \dots u_m v_m \, u_{m + 1} & & u_{k + 1} v_{k + 1} \dots u_m v_m \, u_{m + 1} & & u_{k + 1} v_{k + 1} \dots u_m v_m \, u_{m + 1} & & u_{k + 1} v_{k + 1} \dots u_m v_m \, u_{m + 1} \\
        & \$^{k - 1} & & (\alpha \cdot u_k)^\circ & & (\beta \cdot v_k)^\circ & \\
    };
    \drawArrows{2,4,6,8,10}{2,4,6}
    \draw[dashed, rounded corners] (m-3-2.north -| m-4-1.west) rectangle (m-5-2.south -| m-4-3.east);
    \path[fill=gray, opacity=0.2, rounded corners] (m-10-1.west |- m-6-1.north) rectangle (m-10-7.east |- m-11-6.south);
    \path[fill=gray, opacity=0.2, rounded corners] (m-10-1.west |- m-8-1.north) rectangle (m-10-7.east |- m-11-6.south);
    \path[fill=gray, opacity=0.2, rounded corners] (m-10-1.west |- m-10-1.north) rectangle (m-10-7.east |- m-11-6.south);
  \end{tikzpicture}
  \caption{Cross diagram to distinguish $w$ from $\tilde{w}$ if the first difference is between $u_k v_k$ and $\tilde{u}_k \tilde{v}_k$. Note that $u_k @ \alpha$ and $\alpha \cdot u_k$ are the same as in $\A$ while $v_k @ \beta$ and $v_k \cdot \beta$ are the same as in $\B$, respectively. The dashed part follows from the claim (\ref{eqn:mainTheoremClaim}) and the shaded parts may not exist if the respective parts of $w$ (or $\tilde{w}$) are empty.}\label{fig:caseKgt0}
\end{sidewaysfigure}

Counting the states and letters of the constructed automaton yields the following two corollaries (where we have to take the increased alphabet from the construction for \autoref{fct:nonTrivialAction} into account).
\begin{corollary}
  If the semigroup $S$ is generated by $\A = (P, A, \sigma)$, the semigroup $T$ is generated by $\B = (Q, B, \tau)$ (with $P \cap Q = \emptyset$ and $A \cap B = \emptyset$) and there is a homomorphism $\phi: S \to T$ with $P\phi = \gset{ p \phi }{p \in P} \subseteq Q$, then $S \star T$ is generated by an automaton with state set $P \cup Q$ and an alphabet of size at most $3 + 3 (|A| + 1) + 4 (|B| + 1)$. Furthermore, the construction is computable.
\end{corollary}

\begin{corollary}\label{cor:autSmgrp}
  If $S$ and $T$ are (complete) automaton semigroups such that there is a homomorphism $S \to T$, then $S \star T$ is a (complete) automaton semigroup.
\end{corollary}

\begin{remark}
The earlier symmetric constructions for free products used an
entirely different set of symbols to distinguish alternating products in $S\star T$ depending on 
whether they started with an element of $S$ or with an element of $T$. Furthermore, those 
constructions used symbols $\domino{a}{b}$ ($a\in A, b\in B$) (called \emph{dominoes}) 
with markings similar to the ones used above. 
In the current construction, we mark the letters from $A$ and $B$ directly and are able to 
distinguish all elements by their actions on $\ol{\$}^* A^*$, $\ol{\$}^* B^*$, or on $(B^T)^*$, where $B^T = \{ b^T \mid b \in B \}$.
This difference is, however, purely cosmetic (see \cite{freeprodrevisited} for a version of essentially the same construction
still using the `domino' symbols).

Note that our current construction is more general than \cite[Theorem~3.0.1]{welker} (and thus, in particular, also more general than \cite[Theorem~4]{bc_automaton2}): \cite[Theorem~3.0.1]{welker} states that $S \star T$ is a complete automaton semigroup if there are maps $P \to Q$ and $Q \to P$ which extend into homomorphisms $S \to T$ and $T \to S$, respectively. This hypothesis, for example, cannot be satisfied if $S$ is a free semigroup and $T$ is a finite one. However, a pair of such semigroups clearly satisfies the hypothesis of \autoref{thm:main} (or \autoref{cor:autSmgrp}).
\end{remark}

\section{Corollaries of the Construction}

The hypothesis of \autoref{thm:main} (and \autoref{cor:autSmgrp}) is not very restrictive.\footnote{In fact, it turns out to be rather difficult to find a pair $S, T$ of (finitely generated) semigroups such that there is neither a homomorphism $S \to T$ nor a homomorphism $T \to S$. We will discuss this in \autoref{sct:limits}.}
For example, if at least one of the two semigroups contains an idempotent, we obtain a homomorphism by mapping all elements of the other semigroup to this idempotent, which leads to the following corollary:
\begin{corollary}\label{cor:idempotent}
  Let $S$ and $T$ be (completely) self-similar semigroups such that $S$ or $T$ contains an idempotent. Then $S \star T$ is (completely) self-similar.
If furthermore $S$ and $T$ are (complete) automaton semigroups, then so is $S \star T$.
\end{corollary}

Another large class of semigroups for which the hypothesis of \autoref{thm:main} is satisfied is the class of semigroups admitting a length function. We say that a semigroup $S$ has a \emph{length function} if there is a homomorphism $\lambda: S \to a^+$. The idea here is that the monogenic free semigroup $a^+$ is isomorphic to the semigroup $(\mathbb{N}_{> 0}, +)$ and thus every element $s \in S$ has a \emph{length} $\lambda(s)$. Note that, in particular, free semigroups and free commutative semigroups admit a length function.
\begin{corollary}\label{cor:lengthFunction}
  Let $S$ and $T$ be (completely) self-similar semigroups such that $S$ or $T$ has a length function. Then $S \star T$ is (completely) self-similar.
If furthermore $S$ and $T$ are (complete) automaton semigroups, then so is $S \star T$.
\end{corollary}
\begin{proof}
  Due to symmetry, we may assume without loss of generality that $S$ has a length function, i.\,e.\ that there is a homomorphism $\lambda: S \to a^+$. By choosing an arbitrary element $t \in T$ and mapping $a \mapsto t$, we obtain a homomorphism from $a^+$ to the subsemigroup generated by $t$ in $T$. Composing the two homomorphisms, we obtain a homomorphism $S \to T$ and can apply \autoref{thm:main}
(or \autoref{cor:autSmgrp}).
\end{proof}

\autoref{thm:main} and \autoref{cor:autSmgrp} also generalise by induction to free products of finitely many semigroups,
in the following way.  (Note that the hypothesis is equivalent to the existence of homomorphisms
from $S_i$ to $S_n$ for all $i$.)

\begin{corollary}\label{cor:ind}
Let $S_1,\ldots,S_n$ be a sequence of (completely) self-similar semigroups such that
for $1\leq i\leq n-1$ there is a homomorphism $\phi_i: S_i\ra S_j$ for some $j>i$.
Then the free product $S_1\star S_2\star \ldots \star S_n$ is (completely) self-similar.
If furthermore $S_1,\ldots,S_n$ are (complete) automaton semigroups, then so is their free product.
\end{corollary}
\begin{proof}
The base case $n=2$ is \autoref{thm:main} or \autoref{cor:autSmgrp}.  Suppose the statement is true for 
some $n\geq 2$.  Then $S:= S_2\star S_3\star \ldots \star S_n$ is (completely) self-similar,
and the homomorphism $\phi_1: S_1\ra S_j$ for some $j>1$ extends to a 
homomorphism $\phi: S_1\ra S$, since $S_j$ is a subsemigroup of $S$.
Thus $S_1\star S = S_1\star S_2\star \ldots \star S_n$ is (completely) self-similar by \autoref{thm:main},
and a (complete) automaton semigroup if $S_1,\ldots,S_n$ are by \autoref{cor:autSmgrp}.
Hence, by induction, the statement is true for all $n$.
\end{proof}

In particular, the existence of a single idempotent guarantees that a free product
of arbitrarily many self-similar semigroups is self-similar.

\begin{corollary}
Let $S$ be a free product of finitely many (completely) self-similar semigroups, one of 
which contains an idempotent. Then $S$ is (completely) self-similar.
If furthermore the base semigroups are (complete) automaton semigroups, then so is $S$.
\end{corollary}
\begin{proof}
Write $S = S_1\star S_2\star \ldots \star S_n$ with each $S_i$ a (completely) self-similar semigroup.
Since the free product operation is commutative, we may assume that $S_n$
contains an idempotent $e$.  
Then mapping all elements of $S_i$ to $e$ for $i<n$ gives the homomorphisms
$\phi_i$ in the hypothesis of \autoref{cor:ind}.
\end{proof}

The construction used to prove \autoref{thm:main} can also be used to obtain results which are not immediate corollaries of the theorem (or its corollary for automaton semigroups in \autoref{cor:autSmgrp}). As an example, we prove in the following theorem that it is possible to adjoin a free generator to every self-similar semigroup without losing the self-similarity property and that the analogous statement for automaton semigroups holds as well. The version for automaton semigroups does not follow directly from \autoref{cor:autSmgrp}, as the free monogenic semigroup is not a complete automaton semigroup \cite[Proposition~4.3]{cain_1auto} or even a (partial) automaton semigroup (see \cite[Theorem~18]{structurePart} or \cite[Theorem~1.2.1.4]{waechter2020automaton}).

\begin{theorem}\label{thm:freeGenerator}
  Let $S$ be a (completely) self-similar semigroup. Then $S \star t^+$ is (completely) self-similar. Furthermore, if $S$ is a (complete) automaton semigroup, then so is $S \star t^+$.
\end{theorem}
\begin{proof}
  Let $\A = (P, A, \sigma)$ be a generating automaton for $S$ where no $u \in P^+$ acts as the identity map (which is possible by \autoref{fct:nonTrivialAction}).
  Choose an arbitrary state $\hat{p} \in P$ and let $t \not\in P$ and $0, 1 \not\in A$ be a new state and new letters, respectively. We construct an automaton $\B = (Q, B, \delta)$ generating $S \star t^+$. Our construction is similar to the automaton constructed for \autoref{thm:main}. As the state set we use $Q = P \cup \{ t \}$ and as the alphabet we use
  \[
    B = \gset{ a, a^t, a^S, a^\circ }{ a \in A } \cup \{ 0, 1 \} \cup \{ \$, \hat{\$}, \ol{\$} \} \text{.}
  \]
  As in the proof of \autoref{thm:main}, we speak of $a$ as \emph{unmarked}, $a^t$ and $a^S$ as \emph{$t$-} and \emph{$S$-marked}, respectively, and of $a^\circ$ as \emph{circled}. Additionally, we also apply this notation to strings $\alpha \in A^*$ and refer to the letters $\$$, $\hat{\$}$ and $\ol{\$}$ again as \emph{open}, \emph{half-open} and \emph{closed gates}. The transitions $\delta$ are given by the following table (see also \autoref{fig:schematicFreeGenerator})\footnote{Compare also to the unary adding machine from \autoref{ex:unaryAddingMachine} – although we could also have used the binary adding machine from \autoref{ex:addingMachine}.} where $s\in P, a\in A$ and $s_0$ and $a_0$ are given by $(s, a)\sigma = (s_0, a_0)$, and the respective transitions do not exist (in $\B$) if $\sigma$ is undefined at $(s, a)$:
  \renewcommand{\arraystretch}{1.5}
  \begin{center}
    \begin{tabular}{r|ccccccccc}
          & $0$ & $1$ & $a$ & $a^t$ & $a^S$ & $a^\circ$ & $\$$ & $\hat{\$}$ & $\ol{\$}$ \\\hline
      $s$ & $(s, 0)$ & $(s, 1)$ & $(s, a)$ & $(s_0, a_0^S)$ & $(s_0, a_0^S)$ & $(s, a^\circ)$ & $(s, \$)$ & $(s, \$)$ & $(s, \ol{\$})$ \\
      $t$ & $(\hat{p}, 1)$ & $(t, 1)$ & $(t, a^t)$ & $(t, a^t)$ & $(t, a^\circ)$ & $(t, a^\circ)$ & $(t, \$)$ & $(\hat{p}, \hat{\$})$ & $(\hat{p}, \hat{\$})$
    \end{tabular}
  \end{center}
  Note that $\B$ is finite (complete) if $\A$ is.
  
  \begin{figure}\centering
    \begin{tikzpicture}[auto, shorten >=1pt, >=latex]
      \node[state] (s) {$t$};
      \node[state, right=4.5cm of s] (t) {$s$};
      \node[state, right=2cm of t, dashed] (t_0) {$s_0$};
      \node[state, below=2cm of t, dashed] (that) {$\hat{p}$};

      \draw[->] (s) edge[loop left] node[align=center] (1left)
                      {$1/1$} (s)
                    edge[loop below] node[align=center]
                      {$a^{\phantom{S}} / a^t$\\
                       $a^{\makebox[0pt][l]{$\scriptstyle t$}\phantom{S}} / a^t$\\
                       $a^{S} / a^\circ$\\
                       $a^\circ / a^\circ$} (s)
                    edge[loop above] node (1top) {$\$/\$$} (s)
                    edge node[swap, align=center] (1right) {$\ol{\$}/\hat{\$}$\\$\hat{\$}/\hat{\$}$}
                         node[align=center, pos=0.1] {$0/1$} (that)
      ;
      
      \draw[->] (t) edge[dotted] node[align=center]
                      {${a}^{t} /\, a_0^S$\\
                       ${a}^{S} /\, a_0^S$} (t_0)
                    edge[loop left] node[align=center] (2left)
                      {$0 / 0$\\$1/1$} (t)
                    edge[loop below] node[align=center]
                      {$a / a$\\
                       $a^\circ / a^\circ$} (t)
                    edge[loop above] node[align=center] (2top) {$\ol{\$}/\ol{\$}$\\$\hat{\$}/{\$}$\\$\$/\$$} (t)
      ;

      \path (2left.west) |- coordinate (2topleft) (2top.north);
      \path (that.south) -| coordinate (2bottomright) (t_0.east);
      \draw[rounded corners, dashed] ($(2topleft)+(-0.25cm,0.5cm)$) rectangle ($(2bottomright)+(0.25cm,-0.5cm)$);
      \path ($(2topleft)+(-0.25cm,0.5cm)$) -| node[below left] {$\A$} ($(2bottomright)+(0.25cm,-0.5cm)$);
      
      \path (1left.west) |- coordinate (1topleft) (2topleft);
      \path (1right.east) |- coordinate (1bottomright) (that.south);
      \draw[rounded corners, dashed] ($(1topleft)+(-0.25cm,0.5cm)$) rectangle ($(1bottomright)+(0.25cm,-0.5cm)$);
    \end{tikzpicture}
    \caption{Schematic representation of $\B$. The dotted transitions only exist if the corresponding transitions exist in $\A$.}\label{fig:schematicFreeGenerator}
  \end{figure}
  
  The rest of the proof is also quite similar to the one for \autoref{thm:main}. We show $\Sigma(\B) = S \star t^+$ by proving
  \[
    w =_{S \star t^+} \tilde{w} \iff \forall \beta \in B^*: \beta \cdot w = \beta \cdot \tilde{w} \text{ (or both undefined)}
  \]
  for all $w, \tilde{w} \in (P \cup \{ t \})^+$.
  
  We start with the direction from left to right (i.\,e.\ we show that $\B$ defines an action of $S \star t^+$ on $B^*$). We only need to show something for $w =_{\A} \tilde{w}$ with $w, \tilde{w} \in P^+$ since $\Sigma(\A) = S$ and all relations in $S \star t^+$ stem from relations in $S$ (as $t^+$ is free).
  
  Thus, let $w =_{\A} \tilde{w}$ for $w, \tilde{w} \in P^+$. We only show the implication
  \begin{center}
    \begin{tikzpicture}[>=latex, baseline=(m.base)]
      \matrix (m) [matrix of math nodes,
      text height=\ht\strutbox, text depth=0.25ex,] {
          & b  & \\
        w &    & w' \\
          & b' & \\
      };
      \drawArrows{2}{2}
      
      \matrix[right=of m] (mtilde) [matrix of math nodes,
      text height=\ht\strutbox, text depth=0.25ex,] {
                  & b  & \\
        \tilde{w} &    & \tilde{w}' \textcolor{gray}{{}=_{\A} w'} \\
                  & b' & \\
      };
      \drawArrows[mtilde]{2}{2}
      
      \path (m-2-3.base east) -- node[anchor=base] {$\implies$} (mtilde-2-1.base west);
    \end{tikzpicture}
  \end{center}
  and that $b \cdot w$ is undefined if and only if $b \cdot \tilde{w}$ is
  for all $b \in B$. The rest then follows by induction over the length of $\beta \in B^*$. As $0, 1, a, a^\circ, \$$ and $\ol{\$}$ are ignored\footnote{Remember that we say a symbol $a$ is ignored by a state $p$ if we have $(q, a)\delta = (q, a)$ (i.\,e.\ if there is an $a/a$-labeled self-loop at $q$).} by all states $s \in P$ in $\B$, we only have to show something for $a^t, s^S$ and $\hat{\$}$. Here, we have the cross diagrams
  \begin{center}
    \begin{tikzpicture}[>=latex]
      \matrix (m) [matrix of math nodes,
        text height=\ht\strutbox, text depth=0.25ex,] {
          & \hat{\$} & \\
        w &    & w \\
          & \$ & \\
      };
      \drawArrows{2}{2}
      
      \matrix[right=of m] (aST) [matrix of math nodes,
      text height=\ht\strutbox, text depth=0.25ex,] {
          & {a}^{{\scriptsize t}} \text{ or } {a}^{{\scriptsize S}} & \\
        w &    & w @ a \\
          & {(a \cdot w)}^S & \\
      };
      \drawArrows[aST]{2}{2}
    \end{tikzpicture}
  \end{center}
  (and the corresponding ones for $\tilde{w}$), where the left one settles the case $\hat{\$}$ and, in the right one, the letter at the top may either be $t$- or $S$-marked and $a \cdot w$ and $w @ a$ are as in $\A$ (if they are defined). These cases then follow since $w =_{\A} \tilde{w}$ implies $w@a =_{\A} \tilde{w}@a$ (or both undefined). If $a \cdot w$ and $a \cdot \tilde{w}$ are undefined in $\A$, we have that $a^t \cdot w$, $a^S \cdot w$, $a^t \cdot \tilde{w}$ and $a^S \cdot \tilde{w}$ are also all undefined in $\B$.
  
  For the direction from right to left (i.\,e.\ for showing that the action of $S \star t^+$ on $B^*$ is faithful), we again first claim that we have
  \begin{equation}\label{eqn:freeGeneratorClaim}
    \begin{tikzpicture}[>=latex, baseline=(m-2-3.base)]
    \matrix (m) [matrix of math nodes,
      text height=\ht\strutbox, text depth=0.25ex,] {
        & \ol{\$}^k & \\
      t^{n_1} u_1 \dots t^{n_k} u_k & & \hat{p}^{n_1} u_1 \dots \hat{p}^{n_k} u_k \textcolor{gray}{{}\in P^*} \\
        & \$^k & \\
    };
    \drawArrows{2}{2}
    \end{tikzpicture}
  \end{equation}
  for all $n_1, \dots, n_k > 0$ and $u_1, \dots, u_k \in P^+$. This can be seen from an induction very similar to the one in \autoref{fig:claimInductiveStep}.\footnote{\dots which is not very surprising as we have not changed how we handle the gates.}
  
  Next, we factorize
  \[
    w = u_0 \, t^{n_1} u_1 \dots t^{n_m} u_m \, t^{n_{m + 1}}
  \]
  with $u_0 \in P^*$, $n_1, \dots, n_m > 0$, $u_1, \dots, u_m \in P^+$ and $n_{m + 1} \geq 0$. An analogous factorization of $\tilde{w}$ yields $\tilde{m}$, $\tilde{u}_i$ (for $0 \leq i \leq \tilde{m}$) and $\tilde{n}_i$ (for $1 \leq i \leq \tilde{m} + 1$). Without loss of generality, we may assume $m \leq \tilde{m}$.
  
  First, suppose we have $m < \tilde{m}$ or, equivalently, $\tilde{m} = m + d$ for some $d > 0$. We have
  \begin{center}
    \begin{tikzpicture}[>=latex]
      \matrix (m) [matrix of math nodes,
        text height=\ht\strutbox, text depth=0.25ex,] {
          & \ol{\$}^{m} & & \ol\$^d \\
        u_0 & & u_0 & & u_0 \\
          & \ol{\$}^{m} & & \ol\$^d \\
        t^{n_1} u_1 \dots t^{n_m} u_{m} & & \hat{p}^{n_1} u_1 \dots \hat{p}^{n_m} u_{m} & & \hat{p}^{n_1} u_1 \dots \hat{p}^{n_m} u_{m}\\
          & \$^{m} & & \ol\$^d \\
        t^{n_{m + 1}} & & t^{n_{m + 1}} & & \hat{p}^{n_{m + 1}} \\
          & \$^{m} & & \hat{\$} \, \ol\$^{d - 1} \\
      };
      \drawArrows{2,4,6}{2,4}
      \draw[dashed, rounded corners] (m-3-2.north -| m-4-1.west) rectangle (m-5-2.south -| m-4-3.east);
      \path[fill=gray, opacity=0.2, rounded corners] (m-4-1.west |- m-6-1.north) rectangle (m-4-5.east |- m-7-4.south);
    \end{tikzpicture}
  \end{center}
  where we have used the claim (\ref{eqn:freeGeneratorClaim}) for the dashed part and the shaded part only exists for $n_{m + 1} > 0$. On the other hand, we have $\ol\$^{m + d} \cdot \tilde{w} = \$^{m + d} \cdot t^{n_{m + d + 1}} = \$^{m + d}$ by the claim (\ref{eqn:freeGeneratorClaim}), which is different to $\$^m \, \ol\$^d$ and to $\$^m \, \hat{\$} \, \ol\$^{d - 1}$. Thus, we assume $m = \tilde{m}$ from now on.
  
  The case $u_i \neq_{\A} \tilde{u}_i$ ($u_0$ or $\tilde{u}_0$ may be empty) for some $0 \leq i \leq m$ is similar to what we did in the proof of \autoref{thm:main}. Let $k$ be the minimum of these $i$ and, first, assume $k = 0$, i.\,e.\ we have $u_0 \neq_{\A} \tilde{u}_0$ (including the case that one is $\varepsilon$ while the other one is not). Then there is some $\alpha \in A^+$ (even if one is $\varepsilon$ by our assumption on $\A$) with $\alpha \cdot u_0 \neq \alpha \cdot \tilde{u}_0$ (including the case that exactly one side is undefined). As we are in the case $m = \tilde{m}$, we may assume without loss of generality that $\alpha \cdot u_0$ is defined (due to symmetry in $w$ and $\tilde{w}$), which yields
  \begin{center}
    \begin{tikzpicture}[>=latex]
      \matrix (m) [matrix of math nodes,
      text height=\ht\strutbox, text depth=0.25ex,] {
            & {\alpha}^S & \\
        u_0 & & u_0 @ \alpha \\
            & {(\alpha \cdot u_0)}^S & \\
        t^{n_1} u_1 \dots t^{n_m} u_m \, t^{n_{m + 1}} & & t^{n_1} u_1 \dots t^{n_m} u_m \, t^{n_{m + 1}} \\
            & {(\alpha \cdot u_0)}^{\circ} & \\
      };
      \drawArrows{2,4}{2}
      \path[fill=gray, opacity=0.2, rounded corners] (m-4-1.west |- m-4-1.north) rectangle (m-4-3.east |- m-5-2.south);
    \end{tikzpicture}
  \end{center}
  where the shaded part only exists for $t^{n_1} u_1 \dots t^{n_m} u_m \, t^{n_{m + 1}} \neq \varepsilon$. If $\alpha \cdot \tilde{u}_0$ is undefined in $\A$, we also have that $\alpha^S \cdot \tilde{w}$ is undefined (in $\B$). If $\alpha \cdot \tilde{u}_0$ is defined, we get (by a cross diagram analogous to the one above) ${\alpha}^S \cdot \tilde{w} \in \left\{ {(\beta \cdot \tilde{u}_0)}^S, {(\beta \cdot \tilde{u}_0)}^{\circ} \right\}$ and, thus, $\alpha^S \cdot w \neq \alpha^S \cdot \tilde{w}$ in all cases.
  
  For $k > 0$, there is again some $\alpha \in A^+$ with $\alpha \cdot u_k \neq \alpha \cdot \tilde{u}_k$ in $\A$ (including the case that one is undefined while the other one is not).\footnote{Note that in this case we can neither have $u_k = \varepsilon$ nor $\tilde{u}_k = \varepsilon$.} We may assume that $\alpha \cdot u_k$ is defined and obtain the cross diagram depicted in \autoref{fig:caseKgt0FreeGenerator}, where the gray part does not exist if $t^{n_{k + 1}} u_{k + 1} \dots t^{n_m} u_m \, t^{n_{m + 1}} = \varepsilon$. We also have an analogous diagram for $\tilde{w}$ if $\alpha \cdot \tilde{u}_k$ is defined and obtain
  \[
    \ol{\$}^{k - 1} \alpha \cdot w = \$^{k - 1} (\alpha \cdot u_k)^D \neq \$^{k - 1} (\alpha \cdot \tilde{u}_k)^{\tilde{D}} = \ol{\$}^{k - 1} \alpha \cdot w
  \]
  for some $D, \tilde{D} \in \{ S, \circ \}$. If $\alpha \cdot \tilde{u}_k$ is not defined in $\A$, then neither is $\ol\$^{k - 1} \alpha \cdot \tilde{w}$ in $\B$ and we are done as well.
  
  \begin{sidewaysfigure}\centering
    \begin{tikzpicture}[>=latex]
      \matrix (m) [matrix of math nodes,
      text height=\ht\strutbox, text depth=0.25ex,] {
          & \ol{\$}^{k - 1} & & \alpha & \\
        u_0 & & u_0 & & u_0 \\
          & \ol{\$}^{k - 1} & & \alpha & \\
        t^{n_1} u_1 \dots t^{n_{k - 1}} u_{k - 1} & & \hat{p}^{n_1} u_1 \dots \hat{p}^{n_{k - 1}} u_{k - 1} & & \hat{p}^{n_1} u_1 \dots \hat{p}^{n_{k - 1}} u_{k - 1} \\
          & \$^{k - 1} & & \alpha & \\
        t^{n_k} & & t^{n_k} & & t^{n_k} \\
          & \$^{k - 1} & & \alpha^t & \\
        u_k & & u_k & & u_k @ \alpha \\
          & \$^{k - 1} & & (\alpha \cdot u_k)^S & \\
        t^{n_{k + 1}} u_{k + 1} \dots t^{n_m} u_m \, t^{n_{m + 1}} & & t^{n_{k + 1}} u_{k + 1} \dots t^{n_m} u_m \, t^{n_{m + 1}} & & t^{n_{k + 1}} u_{k + 1} \dots t^{n_m} u_m \, t^{n_{m + 1}} \\
          & \$^{k - 1} & & (\alpha \cdot u_k)^\circ & \\
      };
      \drawArrows{2,4,6,8,10}{2,4}
      \draw[dashed, rounded corners] (m-3-2.north -| m-4-1.west) rectangle (m-5-2.south -| m-4-3.east);
      \path[fill=gray, opacity=0.2, rounded corners] (m-10-1.west |- m-10-1.north) rectangle (m-10-5.east |- m-11-4.south);
    \end{tikzpicture}
    \caption{Cross diagram to distinguish $w$ from $\tilde{w}$ if the first difference arises from $u_k$ and $\tilde{u}_k$. Note that $u_k @ \alpha$ and $\alpha \cdot u_k$ are the same as in $\A$ and the dashed part follows from the claim (\ref{eqn:freeGeneratorClaim}). The shaded part does not exist if the respective part of $w$ (or $\tilde{w}$) is empty. Compare to \autoref{fig:caseKgt0}.}\label{fig:caseKgt0FreeGenerator}
  \end{sidewaysfigure}
  
  It remains the case $n_i \neq \tilde{n}_i$ for some $1 \leq i \leq m + 1$. This time, choose $k$ \textbf{maximal} with this property, i.\,e.\ we have $N_{k + 1} = \sum_{i = k + 1}^{m + 1} n_i = \sum_{i = k + 1}^{m + 1} \tilde{n}_i$ but $n_k \neq \tilde{n}_k$. Without loss of generality, we may assume $n_k < \tilde{n}_k$ and, equivalently, $\tilde{n}_k = n_k + d$ for some $d > 0$. Then, we have $N_{k} = \sum_{i = k}^{m + 1} n_i < \sum_{i = k}^{m + 1} \tilde{n}_i = \tilde{N}_k$ or, more precisely, $\tilde{N}_k = N_k + d$.
  We obtain the cross diagram depicted in \autoref{fig:caseNk} and an analogous diagram for $\tilde{w}$ (compare to the action of the adding machine in \autoref{ex:unaryAddingMachine}). Thus, we have
  \[
    \ol{\$}^{k - 1} 0^{N_k + d} \cdot w = \$^{k - 1} 1^{N_k}0^d \neq \$^{k - 1} 1^{N_k + d} = \$^{k - 1} 1^{\tilde{N}_k} = \ol{\$}^{k - 1} 0^{N_k + d} \cdot \tilde{w} \text{.}\qedhere
  \]
  \begin{sidewaysfigure}\centering
    \begin{tikzpicture}[>=latex]
      \matrix (m) [matrix of math nodes,
      text height=\ht\strutbox, text depth=0.25ex,] {
          & \ol{\$}^{k - 1} & & 0^{N_k + d} & \\
        u_0 & & u_0 & & u_0 \\
          & \ol{\$}^{k - 1} & & 0^{N_k + d} & \\
        t^{n_1} u_1 \dots t^{n_{k - 1}} u_{k - 1} & & \hat{p}^{n_1} u_1 \dots \hat{p}^{n_{k - 1}} u_{k - 1} & & \hat{p}^{n_1} u_1 \dots \hat{p}^{n_{k - 1}} u_{k - 1}\\
          & \$^{k - 1} & & 0^{N_k + d} & \\
        t^{n_k} & & t^{n_k} & & \hat{p}^{n_k} \\
          & \$^{k - 1} & & 1^{n_k} 0^{N_k + d - n_k} & \\
        u_k & & u_k & & u_k \\
          & \$^{k - 1} & & 1^{n_k} 0^{N_k + d - n_k} & \\
        t^{n_{k + 1}} u_{k + 1} \dots t^{n_m} u_m \, t^{n_{m + 1}} & & t^{n_{k + 1}} u_{k + 1} \dots t^{n_m} u_m \, t^{n_{m + 1}} & & \hat{p}^{n_{k + 1}} u_{k + 1} \dots \hat{p}^{n_m} u_m \hat{p}^{n_{m + 1}} \\
          & \$^{k - 1} & & 1^{N_k}0^d & \\
      };
      \drawArrows{2,4,6,8,10}{2,4}
      \draw[dashed, rounded corners] (m-3-2.north -| m-4-1.west) rectangle (m-5-2.south -| m-4-3.east);
    \end{tikzpicture}
    \caption{Cross diagram to distinguish $w$ from $\tilde{w}$ if the first difference is at $n_k \neq \tilde{n}_k$ (compare to \autoref{ex:unaryAddingMachine}). The dashed part follows from the claim (\ref{eqn:freeGeneratorClaim}).}\label{fig:caseNk}
  \end{sidewaysfigure}
\end{proof}

Having \autoref{thm:freeGenerator} to handle the special case of the free semigroup of rank one allows us to make more natural statements. For example, we may now state:
\begin{corollary}\label{cor:finiteOrFree}
  Let $S$ be a (completely) self-similar semigroup and let $T$ be a finite or free semigroup. Then $S \star T$ is (completely) self-similar. If furthermore $S$ is a (complete) automaton semigroup, then so is $S \star T$.
\end{corollary}
\begin{proof}
  If $T$ is finite, it is a complete automaton semigroup \cite[Proposition~4.6]{cain_1auto} (and, thus, in particular completely self-similar). In addition, it must contain an idempotent and we obtain the statement by \autoref{cor:idempotent}.
  
  If $T$ is free of rank at least two, it is also a complete automaton semigroup (see \autoref{ex:freeSemigroup} and \cite[Proposition~4.1]{cain_1auto}) and admits a length function. The statement for this case then follows from \autoref{cor:lengthFunction}.
  
  Finally, if $T$ is free of rank one, we exactly have the statement of \autoref{thm:freeGenerator}.
\end{proof}

\section{Limits of the Construction}\label{sct:limits}

By Corollaries~\ref{cor:idempotent} and~\ref{cor:lengthFunction}, we have to look into idempotent-free automaton semigroups without length functions in order to find a pair of self-similar (or automaton) semigroups not satisfying the hypothesis of \autoref{thm:main} (or \autoref{cor:autSmgrp}), which would be required in order to either relax the hypothesis even further (possibly with a new construction), or provide a pair $S, T$ of self-similar semigroups such that $S \star T$ is not self-similar.  It turns out that we can reduce the class of potential candidates even further: we will show next that no finitely generated simple or $0$-simple idempotent-free semigroup is self-similar (and, thus, that no simple or $0$-simple idempotent-free semigroup is an automaton semigroup).

In the following, we write $S^1$ for the monoid arising from a semigroup $S$ by adjoining an identity element if $S$ does not already contain one (if $S$ is a monoid, we have $S^1 = S$).
For two subsets $X, Y$ of a semigroup, we write $XY = \{ xy \mid x \in X, y \in Y \}$ and observe that this inherits the associativity of the product in the semigroup.

\paragraph*{\textbf{Simple and $0$-Simple Idempotent-Free Semigroups.}}
A (two-sided) \emph{ideal} of a semigroup $S$ is a subset $I \subseteq S$ with $S^1 I S^1 \subseteq I$.
A semigroup $S$ is \emph{simple} if $\emptyset$ and $S$ are its only ideals. A \emph{zero} of a semigroup $S$ is an element $0 \in S$ with $s0 = 0s = 0$ for all $s \in S$. If a semigroup $S$ contains a zero, it is unique and we denote it by $0$. Of course, $\{ 0 \}$ is always an ideal if a zero exists. A semigroup $S$ with zero is called \emph{$0$-simple} if $\emptyset, \{ 0 \}$ and $S$ are its only ideals and we have $S^2 \neq \{ 0 \}$, i.\,e.\ $S$ is not a null semigroup 
(the latter condition is a technical requirement, see \cite[p.~66]{HOWIE} for further details).

A semigroup with zero is called \emph{idempotent-free} if $0$ is its only idempotent and a semigroup without zero is called \emph{idempotent-free} if it does not contain an idempotent.

A special role in the theory of simple and $0$-simple idempotent-free semigroups is played by the semigroups $A = \langle a, b \mid a^2 b = a \rangle$ and $C = \langle a, b \mid a^2 b = a, a b^2 = b \rangle$ (see \cite{jones_bicyclic} for a further discussion):
\begin{proposition}[follows from {\cite[Corollary~5.2]{jones_bicyclic}}]\label{prop:ACembedding}
  Let $S$ be a finitely generated simple or $0$-simple idempotent-free semigroup. Then, there is an injective homomorphism $A \to S$ or an injective homomorphism $C \to S$.
\end{proposition}
In particular, $A$ and $C$ are both idempotent-free.

If there is an injective homomorphism $S \to T$ between two semigroups $S$ and $T$, then $S$ is isomorphic to a subsemigroup of $T$ and we also say that $S$ \emph{embeds into} $T$ (or that $S$ can be \emph{embedded into} $T$). Thus, \autoref{prop:ACembedding} states that, into every finitely generated simple or $0$-simple idempotent-free semigroup, we can embed $A$ or $C$.

\paragraph*{\textbf{Residually Finite Semigroups.}}
For our proof, we will show that no simple or $0$-simple idempotent-free semigroup is residually finite. A semigroup $S$ is called \emph{residually finite}, if, for all $s, t \in S$ with $s \neq t$, there is a homomorphism $\varphi: S \to F$ from $S$ to some finite semigroup $F$ with $s\varphi \neq t\varphi$. Clearly, every subsemigroup of a residually finite semigroup is also residually finite:
\begin{fact}\label{fct:residuallyFiniteSubsemigroup}
  If a semigroup $S$ is not residually finite and embeds into a semigroup $T$, then $T$ cannot be residually finite either.
\end{fact}

Every complete automaton semigroup is residually finite \cite[Proposition~3.2]{cain_1auto} and the argument can easily be extended to general self-similar semigroups.
\begin{fact}\label{fct:selfSimilarIsResiduallyFinite}
  Every self-similar semigroup $S$ is residually finite.
\end{fact}
\begin{proof}
  Let $S$ be generated by $\A = (Q, A, \delta)$ and let $u, v \in Q^+$ with $u\neq_{\A} v$. By definition, there is some string $\alpha \in A^+$ of length $n$ with $\alpha \cdot u \neq \alpha \cdot v$ (including the case that one is defined while the other one is not). By restricting the action of $Q^+$ on $A^+$ to an action of $Q^+$ on $A^n$ (which is possible because the action of $\A$ is 
length-preserving), we obtain a finite quotient $F$ of $S$ where $u$ and $v$ are different elements (as they act differently on $\alpha \in A^n$).
\end{proof}

We will show that neither $A$ nor $C$ is residually finite and finally combine all the pieces.
The arguments are the same as in \cite[Example~4.7]{lallement1974monoids} but we include a proof for completeness.
\begin{proposition}\label{prop:ACnotResiduallyFinite}
  Neither $A = \langle a, b \mid a^2 b = a \rangle$ nor $C = \langle a, b \mid a^2 b = a, a b^2 = b \rangle$ is residually finite.
\end{proposition}
\begin{proof}
  The proof is the same for $A$ and $C$. Therefore, let $X = A$ or $X = C$. Since $X$ is idempotent-free, we have $ab \neq (ab)^2$ in $X$. However, we show that any homomorphism $\varphi: X \to F$ to a finite semigroup $F$ maps $ab$ and $(ab)^2$ to the same element.

  Fix such a homomorphism and, for simplicity, write $a$ for $a\varphi$ and $b$ for $b\varphi$.
  All further calculations are in $F$. Note that the relation $a\, ab = a$ holds here as well.
  
  By the pigeon-hole principle, there is some minimal $m \geq 1$ and some minimal $n > m$ with $a^m = a^n$ (since $F$ is finite).
  For $m > 1$ (and, thus, $n > 2$), we have
  \begin{align*}
    a^{m - 1} &= a^{m - 2} \, a = a^{m - 2} \, a^2 b = a^m b = a^n b = a^{n - 2} \, a^2 b = a^{n - 2} \, a = a^{n - 1} \text{,}
  \intertext{which contradicts the minimality of $m$ and $n$. For the remaining case $m = 1$ and $n > 1$ (i.\,e.\ $a = a^n$), we have}
    a^{n - 1} &= a^{n - 2} \, a = a^{n - 2} \, a^2 b = a^{n} b = ab
  \quad\text{and, thus,}\quad
    a = a^n = aba \text{,}
  \end{align*}
  which implies $(ab)^2 = ab$ as desired.
\end{proof}

\begin{theorem}\label{thm:simpleIdempotentFreeIsNotResiduallyFinite}
  Let $S$ be a finitely generated simple or $0$-simple idempotent-free semigroup. Then $S$ is not residually finite.
\end{theorem}
\begin{proof}
  By \autoref{prop:ACembedding}, $A$ or $C$ embeds into $S$. Since neither of the two is residually finite (by \autoref{prop:ACnotResiduallyFinite}), we obtain that $S$ cannot be residually finite either by \autoref{fct:residuallyFiniteSubsemigroup}.
\end{proof}

Since self-similar semigroups are residually finite (\autoref{fct:selfSimilarIsResiduallyFinite}),
we have:

\begin{corollary}\label{cor:simpleIdempotentIsNotAutomaton}
  A finitely generated simple or $0$-simple idempotent-free semigroup is not self-similar. In particular, a simple or $0$-simple idempotent-free semigroup is not an automaton semigroup.
\end{corollary}

\paragraph*{\textbf{A Pair of Semigroups Without Homomorphisms.}\protect\footnote{The authors would like to thank Emanuele Rodaro for his help in finding this example.}}
We conclude this section by presenting a pair $S, T$ of semigroups without a homomorphism\ $S \to T$ or $T \to S$ where $S$ and $T$ possess typical properties of automaton semigroups, which makes them good candidates for also belonging to this class (and therefore interesting in the light of \autoref{thm:main} and \autoref{cor:autSmgrp}): $S$ and $T$ will be
\begin{itemize}
  \item finitely generated (automaton semigroups are generated by the finitely many state of the generating automaton),
  \item residually finite (like all self-similar semigroups; see \autoref{fct:selfSimilarIsResiduallyFinite}) and 
  \item with word problem solvable in linear time.
\end{itemize}
The word problem of a semigroup finitely generated by some set $Q$ is the decision problem whether two input words over $Q$ represent the same semigroup element. The word problem of any automaton semigroup can be solved in polynomial space \cite{dangeli2017complexity} and, under common complexity theoretic assumptions, this cannot significantly be improved as there is an automaton group whose word problem is hard for this complexity class \cite{waechter2023automaton}.

Despite having all these properties, it is not clear whether either of the semigroups $S$ and $T$ is an automaton semigroup (or at least self-similar).

We choose $S = \langle a, b \mid ba = b \rangle$. Clearly, $S$ is finitely generated and
it is easy to see that every element of $S$ can be written in normal form as $a^i b^k$ or $b^i$ with $i > 0$ and $k \geq 0$.
Since the normal form can be computed in linear time (in fact, it may even be computed by a rational transducer), we immediately obtain that the word problem of $S$ can also be solved in linear time.
The product of two elements in normal form (with $i + k > 0$ and $j + \ell > 0$) is given by:
\begin{equation*}
  a^i b^k \, a^j b^\ell = \begin{cases}
    a^{i} b^{k + \ell} & \text{if } k > 0 \text{,}\\
    a^{i + j} b^\ell & \text{otherwise} \text{.}
  \end{cases}
\end{equation*}
Thus, for $i = j$ and $k = \ell$, the normal form of the product will either have at least one additional $a$ or on additional $b$, which shows:
\begin{fact}\label{fct:SidempotentFree}
  $S = \langle a, b \mid ba = b \rangle$ is idempotent-free.
\end{fact}

Similarly, we may also show that $S$ is right cancellative (i.\,e.\ that $st = s't$ implies $s = s'$ for all $s, s', t \in S$).
\begin{fact}\label{fct:SrightCancellative}
  $S = \langle a, b \mid ba = b \rangle$ is right cancellative.
\end{fact}
\begin{proof}
  We may consider the elements in normal form and assume $a^i b^k \, a^{j} b^\ell = a^{i'} b^{k'} \, a^j b^\ell$ (for $i' + k' > 0$, $i + k > 0$ and $j + \ell > 0$).
  
  Suppose we have $k = 0$ but $k' > 0$. On the left-hand side, we get $a^{i + j}b^\ell$ and, on the right-hand side, we get $a^{i'} b^{k' + \ell}$. Since they must be equal, we obtain $\ell = k' + \ell$ and, thus, $k' = 0$; a contradiction. The situation $k > 0$ but $k' = 0$ is symmetrical.
  
  Thus, for $k = k' = 0$, we need to have $a^{i + j}b^\ell = a^{i' + j}b^\ell$, which implies $i = i'$, and are done. For $k, k' > 0$, we have $a^{i} b^{k + \ell} = a^{i'} b^{k' + \ell}$ and, thus, $i = i'$ and $k + \ell = k' + \ell$. The latter implies $k = k'$.
\end{proof}

With respect to the above properties, we point out that $S$ is residually finite by \cite[Example~4.6]{lallement1974monoids}.\footnote{Strictly, speaking this shows that $S^\textnormal{op}$ is residually finite; see below.}
For the sake of completeness, we give an alternative proof here.
\begin{proposition}\label{prop:SnotResiduallyFinite}
  $S = \langle a, b \mid ba = b \rangle$ is residually finite.
\end{proposition}
\begin{proof}
  Let $s, t \in \{ a^i b^k, b^i \mid i > 0, k \geq 0 \}$ be elements of $S$ in normal form with $s \neq t$. We need some finite semigroup $F$ and a homomorphism $\varphi: S \to F$ with $s\varphi \neq t\varphi$.
Let $m$ be the number of $b$s in $s$ and $n$ the number of $b$s in $t$. 

If $m\neq n$, assume $m < n$ without loss of generality, and let $F$ be the cyclic group of order $n + 1$, generated by $x$, with identity $1$.  Observe that $a \mapsto 1$, $b \mapsto x$ induces a well-defined 
homomorphism $\varphi: S \to F$, and that $s \varphi = x^m \neq x^n = t \varphi$.

If $m=n$, then $s$ and $t$ cannot both be of the form $b^i$. Without loss of generality, 
write $s = a^i b^m$ and $t = a^j b^m$ for $0 \leq i < j$. 
Let $F$ be the transition semigroup\footnote{In a deterministic and complete finite automaton, every letter $a$ induces a transition function, which maps a state $q$ to the (unique and existing) state reached from $q$ by reading an $a$. The closure of these functions under composition forms the \emph{transition semigroup} of the automaton (see \cite{lawson2004finite} for details).} of the automaton (without output) depicted in \autoref{fig:transitionSemigroup}. Its state set is the union of $\{ 0, \dots, i, \dots, j \}$ and its disjoint copy $\{ 0', \dots, i', \dots, j' \}$. The homomorphism $\varphi: S \to F$ is given by mapping $a$ to the transition function belonging to $a$ in the automaton and mapping $b$ to the transition function belonging to $b$, respectively. The reader may observe that reading $ba$ in the automaton always ends in the same state as reading $b$. Thus, the mapping indeed extends into a well-defined homomorphism of semigroups. When we start in state $0$ and read $a^i b^m$, we can either end in state $i$ or in state $i'$. Similarly, we end either in $j$ or in $j'$ if we read $a^j b^m$. As we have $i \neq j$, this shows $s \varphi \neq t \varphi$.
\end{proof}
\begin{figure}\centering
  \begin{tikzpicture}[auto, shorten >=1pt, >=latex]
    \node[state] (0) {$0$};
    \node[state, below=of 0] (0') {$0'$};

    \node[state, right=of 0] (1) {$1$};
    \node[state, below=of 1] (1') {$1'$};
    
    \node[right=of 1] (dotsl) {$\dots$};
    \node[right=of 1'] (dotsl') {$\dots$};
    
    \node[state, right=of dotsl] (i) {$i$};
    \node[state, below=of i] (i') {$i'$};
    
    \node[right=of i] (dotsr) {$\dots$};
    \node[right=of i'] (dotsr') {$\dots$};
    
    \node[state, right=of dotsr] (j) {$j$};
    \node[state, below=of j] (j') {$j'$};
    
    \path[->] (0) edge node {$a$} (1)
                  edge node {$b$} (0')
              (1) edge node {$a$} (dotsl)
                  edge node {$b$} (1')
              (dotsl) edge node {$a$} (i)
              (i) edge node {$a$} (dotsr)
                  edge node {$b$} (i')
              (dotsr) edge node {$a$} (j)
              (j) edge[loop right] node {$a$} (j)
                  edge node {$b$} (j')
              (0') edge[loop below] node {$a, b$} (0')
              (1') edge[loop below] node {$a, b$} (1')
              (i') edge[loop below] node {$a, b$} (i')
              (j') edge[loop below] node {$a, b$} (j')
    ;
  \end{tikzpicture}
  \caption{Automaton whose transition semigroup $F$ is used to distinguish $a^i b^m$ and $a^j b^m$.}\label{fig:transitionSemigroup}
\end{figure}

\paragraph*{\textbf{The Opposite Semigroup.}}
The \emph{opposite semigroup} $S^\textnormal{op}$ of an arbitrary semigroup $S$ with operation $s \cdot t$ is given by the same elements as in $S$ but with the operation $s * t = t \cdot s$. A simple calculation shows that a homomorphism $\varphi: S \to T$ naturally induces a homomorphism $S^\textnormal{op} \to T^\textnormal{op}$. In particular, we obtain:
\begin{fact}\label{fct:opResiduallyFinite}
  A semigroup $S$ is residually finite if and only if $S^\textnormal{op}$ is residually finite.
\end{fact}

For our pair of finitely generated and residually finite semigroups without homomorphisms $S \to T$ and $T \to S$, we choose the previously discussed semigroup $S = \langle a, b \mid ba = b \rangle$ and its opposite 
$T = S^\textnormal{op} = \langle a, b\mid ab = b\rangle$. Due to \autoref{prop:SnotResiduallyFinite} and \autoref{fct:opResiduallyFinite}, it only remains to show that there are no homomorphisms in either direction 
between $S$ and $T$.
\begin{proposition}\label{prop:SHasNoHoms}
  There is no homomorphism between $S = \langle a, b \mid ba = b \rangle$ and its opposite $S^\textnormal{op}$.
\end{proposition}
\begin{proof}
  We show the statement by contradiction.
  
  If there is a homomorphism $S \to S^\textnormal{op}$, there is also a homomorphism $S^\textnormal{op} \to (S^\textnormal{op})^\textnormal{op} = S$. Thus, consider a homomorphism $\varphi: S^\textnormal{op} \to S$.
  
  In $S^\textnormal{op}$, we have $a * b = ba = b$ and, thus,
  \[
    b\varphi = (a * b)\varphi = (a\varphi)(b\varphi)
  \]
  where we have used that $\varphi$ is a homomorphism. Multiplying this equation on the left by $a \varphi$, we obtain $(a \varphi)(b\varphi) = (a\varphi)^2 (b\varphi)$. Since $S$ is right cancellative (by \autoref{fct:SrightCancellative}), we obtain $a \varphi = (a \varphi)^2$. However, this is impossible since $S$ is idempotent-free (by \autoref{fct:SidempotentFree}).
\end{proof}

\begin{question}\label{q:SSelfSimilar}
  Is the semigroup $S$ self-similar? Is it an automaton semigroup? What about $S^\textnormal{op}$?
\end{question}

Note that it is not known whether the class of automaton semigroups is closed under taking the opposite semigroup \cite[Question~13]{bc_automaton2}.
In defining automaton semigroups, we make a choice as to whether states act on strings on the right (as in this paper) or the left,
and it is not clear whether this choice actually makes a difference to the class of semigroups defined.

\section{Future work}
The existence of a homomorphism from one base semigroup to another appears
fundamental to the construction idea used in all results about free products of 
automaton semigroups (or self-similar semigroups) to date.  If the hypothesis of Theorem~\ref{thm:main} is
not in fact necessary, it is likely that either the new examples will be in some sense
`artificial' or that a novel construction idea will be required.

The undecidability of the finiteness problem \cite{gillibert_finiteness} and freeness problem
\cite{freenessProblem} for automaton semigroups lead us to conjecture that 
the `automaton-semigroupness' problem for free products of automaton semigroups may be undecidable
-- unless it turns out to be trivial --
considering the statement of \autoref{cor:finiteOrFree} that $S\star T$ is an automaton semigroup
if $S$ is any automaton semigroup and $T$ is finite or free (in particular of rank at least two).

 \begin{conjecture}
 One of the following holds:
 \begin{itemize}
 \item Every free product of two (and hence of finitely many) automaton semigroups is itself an automaton semigroup; or
\item It is undecidable, given two automata $\A$ and $\B$, whether the free 
product $\Sigma(\A)\star \Sigma(\B)$ is an automaton semigroup.
\end{itemize}
\end{conjecture}

\section*{Acknowledgements}
The first author was supported by the Funda\c{c}\~{a}o para a Ci\^{e}ncia e a Tecnologia (Portuguese Foundation for Science and Technology) through an {\sc FCT} post-doctoral fellowship ({\sc SFRH}/{\sc BPD}/121469/2016) and the projects {\sc UID}/{\sc MAT}/00297/2013 (Centro de Matem\'{a}tica e Aplica\c{c}\~{o}es) and {\sc PTDC}/{\sc MAT-PUR}/31174/2017.

During the research and writing for this paper, the second author was previously affiliated with Institut für Formale Methoden der Informatik (FMI) at Universität Stuttgart, Centro de Matemática da Universidade do Porto (CMUP), which is financed by national funds through FCT – Fundação para a Ciência e Tecnologia, I.P., under the project with reference UIDB/00144/2020, the Dipartimento di Matematica of the Politecnico di Milano where he was funded by the Deutsche Forschungsgemeinschaft (DFG, German Research Foundation) – 492814705 and the Fachrichtung Mathematik of the Universität des Saarlandes, partly funded by ERC grant 101097307.
The listed affiliation is his current one where this work was supported by the Engineering and Physical Sciences Research Council [grant number EP/Y008626/1].

The authors would like to thank the anonymous referee for pointing out simplifications of the arguments in \autoref{sct:limits}.

\bibliographystyle{plainurl}
\bibliography{references}

\end{document}